\documentclass[reqno,twoside,11pt]{amsart}
\usepackage{amsmath, amsfonts, amssymb, esint, amsthm, nicefrac, bbm, dsfont}
\usepackage{epsfig, graphicx,xcolor}

\newcommand{\dx}{~\mathrm{d}x}
\newcommand{\sym}{\mathrm{sym}}

\newcommand{\Id}{\mathrm{Id}}

\newcommand{\R}{\mathbb{R}}
\def\endproof{\hspace*{\fill}\mbox{\ \rule{.1in}{.1in}}\medskip }
\newcommand*{\dbar}[1]{\bar{\bar{#1}}}
\newcommand*{\tbar}[1]{\bar{\dbar{#1}}}

\numberwithin{equation}{section}
\theoremstyle{plain}

\newtheorem*{theorem*}{Theorem}
\newtheorem{theorem}{Theorem}[section]
\newtheorem{lemma}[theorem]{Lemma}
\newtheorem{corollary}[theorem]{Corollary}

\theoremstyle{definition}

\begin{document}
\title [The Monge-Amp\`ere system in codimension $d_*-d+1$]
{Full flexibility of the Monge-Amp\`ere system in 
  codimension $d_*-d+1$}
\author{Wentao Cao, Jonas Hirsch, Dominik Inauen and Marta Lewicka}
\address{W.C.: Academy for Multidisciplinary Studies, Capital Normal University, West 3rd Ring North Road 105, Beijing, 100048 P.R. China}
\address{J.H.: Institut f\"{u}r Mathematik, Universit\"{a}t Leipzig, D-04109, Leipzig, Germany}
\address{D.I.: Universit\'e de Fribourg, CH-1700, Fribourg, Switzerland}
\address{M.L.: University of Pittsburgh, Department of Mathematics, 
139 University Place, Pittsburgh, PA 15260, USA}
\email{cwtmath@cnu.edu.cn, jonas.hirsch@math.uni-leipzig.de, 
dominik.inauen@unifr.ch, lewicka@pitt.edu} 

\date{\today}
\thanks{AMS classification: 35J96, 53C42, 53A35\\
W.C. was supported in part by the National Natural Science Foundation of China (No. 12471224). 
M.L. was partially supported by NSF grant DMS-2006439.
}

\begin{abstract} 
We prove that $\mathcal{C}^{1,\alpha}$ solutions to the Monge-Amp\`ere
system in dimension $d$ and codimension $k= d_*-d+1$, where
$d_*$ denotes the Janet dimension, are dense in the space
of continuous functions, for every H\"older exponent $\alpha<1$.
Our result strengthens the statement in \cite{lew_conv}, obtained for
$k = 2d_*$ and based on ideas from \cite{Kallen} in the context of the
isometric immersion system. It also generalizes the result of
\cite{in_lew2}, where full flexibility was established in dimension
$d=2$ and codimension $k=2$.
The same proof scheme further yields local full flexibility of
isometric immersions of $d$-dimensional Riemannian metrics into
Euclidean space of dimension $d_* + 1$, generalizing the result in
\cite{lew_full2d2k} proved for $d=k=2$. By using techniques of \cite{CDS}, the result can then be extended to compact manifolds, in codimension $(d+1)d_*-d+1$. 
\end{abstract}

\maketitle

\section{Introduction}\label{sec_intro}

In this paper, we prove a new {\em flexibility} result
for the Monge-Amp\`ere system, in arbitrary dimension $d$, and with the
specific target codimension $k$:
$$ k = \bar d \quad \mbox{ where } \quad \bar d\doteq \frac{d^2-d+2}{2}.$$
We note the relation of $\bar d$ to the Janet dimension $d_*= d(d+1)/2$:
$$\bar d= d_*-d+1.$$
In particular, $\bar d$ is significantly smaller than $2d_*$, 
that is the codimension in which {\em full flexibility}, in the sense of Theorem \ref{th_final} below, was previously established.
More precisely, we prove:
\begin{theorem}\label{th_final}
Let $\omega\subset\R^d$ be an open, bounded, $d$-dimensional
set. Let  the target codimension:
$$\bar d = d_*-d+1 \quad \mbox{where} \quad d_*=\frac{d(d+1)}{2}.$$
Given vector fields $v\in\mathcal{C}^1(\bar\omega,\R^{\bar d})$,
$w\in\mathcal{C}^1(\bar\omega,\R^{d})$
and a matrix field $A\in\mathcal{C}^{r,\beta}(\bar\omega,\R^{d\times d}_\sym)$,
assume:
\begin{equation*}
\frac{1}{2}(\nabla v)^T\nabla v + \sym\nabla
w < A \quad \; \mbox{ in } \; \bar\omega,
\end{equation*}
in the sense of matrix inequalities. Then, for every exponent $\alpha$, satisfying:
\begin{equation}\label{VKrange} 
\alpha<\min\Big\{\frac{r+\beta}{2}, 1 \Big\},
\end{equation}
and for every $\epsilon>0$,  there exist $\tilde v\in\mathcal{C}^{1,\alpha}(\bar\omega,\R^{\bar d})$,
$\tilde w\in\mathcal{C}^{1,\alpha}(\bar\omega,\R^{d})$ such that:
\begin{align*}
& \|\tilde v - v\|_0\leq \epsilon, \qquad \|\tilde w - w\|_0\leq \epsilon,
\nonumber \vspace{1mm}\\ 
& \frac{1}{2}(\nabla \tilde v)^T\nabla \tilde v + \sym\nabla
\tilde w = A \quad \mbox{ in }\;\bar\omega. \nonumber
\end{align*}
\end{theorem}

\smallskip

\noindent Recall that the Monge-Amp\`ere system, introduced in \cite{lew_conv},
seeks a $k$-dimensional solution $v$ to:
\begin{equation}\label{MA}
\left\{\begin{split}
& v:(\omega\subset\R^d)\to \R^k, 
\\ & \mathfrak{Det} \,\nabla^2v \doteq \big[\langle \partial_{is}
v, \partial_{jt}v\rangle -\langle \partial_{it}v, \partial_{js}v\rangle\big]_{i,j,
  s,t:1\ldots d}= F\quad\mbox{ in }\;\omega,
\end{split}\right.
\end{equation}
where $F:\omega\to\R^{d^4}$ is a given table field that satisfies
symmetry conditions (\ref{symF}), necessary for the resolvability of (\ref{MA}).
The first main reason why the system (\ref{MA}) is worthy of
attention, is that it provides the multi-dimensional version of the Monge-Amp\`ere equation:
\begin{equation}\label{MAeq}
\left\{\begin{split}
& \, v:(\omega\subset\R^2)\to \R, 
\\ & \det\nabla^2 v  \doteq \partial_{11}
v  \partial_{22} v - (\partial_{12} v)^2 = f\quad\mbox{ in }\;\omega,
\end{split}\right.
\end{equation} 
arising from the prescribed curvature problem. Indeed, the family of Riemannian metrics on $\omega$, generated by immersions $\{u^\epsilon=(id_d, \epsilon
v)\}_{\epsilon\to 0}$, have their Riemann curvatures satisfy:
\begin{equation*}
{Riem}\big((\nabla u^\epsilon)^T\nabla u^\epsilon\big)  =
{Riem}\big(\mbox{Id}_d + \epsilon^2(\nabla v)^T\nabla v\big) 
= \epsilon^2 \mathfrak{Det} \,\nabla^2v + o(\epsilon^2).
\end{equation*}
When $d=2$, $k=1$ as in (\ref{MAeq}), the above is consistent with the
formula for the Gaussian curvature of surfaces described as the 
graph of $\{\epsilon v\}_{\epsilon\to 0}$, namely:
$$\kappa(\epsilon v) = \frac{\epsilon^2 \det\nabla^2 v}{(1+\epsilon^2|\nabla
  v|^2)^2} = \epsilon^2 \det\nabla^2v + o(\epsilon^2),$$
which yields (\ref{MAeq}) as the prescription of Gaussian curvature.
The second reason to study (\ref{MA}), is related to its weak formulation, given in the form of the
following Von K\`arm\`an system:
\begin{equation}\label{VK}
\left\{\begin{split}
&\,  v:\omega\to \R^k, \quad w:\omega\to\R^d,
\\ & \frac{1}{2}(\nabla v)^T\nabla v + \sym \nabla w = A\quad\mbox{ in }\;\omega.
\end{split}\right.
\end{equation}
Here, one seeks a solution pair, consisting of
$k$-dimensional $v$ and $d$-dimensional $w$, for a given matrix field
$A:\omega\to\R^{d\times d}_\sym$ determined by $F$. 
The system (\ref{VK}) arises 
in the context of finding an isometric immersion $u$ of the given Riemannian metric
$g:\omega\to\R^{d\times d}_{\sym, >}$, into $\R^{d+k}$:
\begin{equation}\label{II}
\left\{\begin{split}
& \, u:(\omega\subset\R^d)\to \R^{d+k},
\\ & (\nabla u)^T\nabla u = g\quad\mbox{ in }\;\omega.
\end{split}\right.
\end{equation}
Indeed, (\ref{II}) yields (\ref{VK}) when equating the leading order terms in
the family of Riemannian metrics $\{\Id_d+2\epsilon^2 A\}_{\epsilon\to 0}$
and the metrics generated by the immersions $\{\bar u^\epsilon =
({id}_d+\epsilon^2w, \epsilon v)\}_{\epsilon\to 0}$:
$$(\nabla \bar u^\epsilon)^T\nabla \bar u^\epsilon = \mbox{Id}_d +
\epsilon^2 \big((\nabla v)^T\nabla v + 2\, \sym\nabla w\big) + o(\epsilon^2).$$
When $d=2$, $k=1$, the above derivation is consistent with 
the left hand side of (\ref{VK}) appearing in the stretching energy of
a thin film, subject to the out-of-plane displacement $v$ and the in-plane displacement $w$.  
Hence, the three problems (\ref{MA}), (\ref{VK}) and
(\ref{II}) are intrinsically related.

\bigskip

\noindent We remark that the results as in Theorem \ref{th_final} (as
well as Corollary \ref{th_weakMA} below) remain valid in any target space $\mathbb{R}^k$ replacing $\R^{\bar d}$, provided that $k \geq\bar d$. Since our proofs rely only on corrugations in $\bar d$ distinct directions, the same
construction of {\em stage} as in Theorem \ref{thm_stage} applies.
Another observation is that by extending our proofs - at the expense of additional technical constructions outlined in Section \ref{sec_iso_imm} - analogous full flexibility results hold in the setting of isometric immersions, that is for solutions of system (\ref{II}), posed on $\omega \subset \mathbb{R}^d$ and with target dimension $d + \bar d = d_* + 1$, as stated in Theorem \ref{th_final2}. In particular, this establishes the local existence of $\mathcal{C}^{1,1-}$-regular isometric immersions of $d$-dimensional $\mathcal{C}^2$ Riemannian metrics into the $(d_*+1)$-dimensional Euclidean space. By using techniques of \cite{CDS}, the result can then be extended to compact manifolds, in codimension $k\geq (d+1)d_*-d+1$.

\medskip

\noindent As a consequence and in light of \cite[Theorem 1.8, Theorem 1.10]{lew_conv}, we further obtain: 

\begin{corollary}\label{th_weakMA} 
Let $F=[F_{ij, st}]_{i,j,s,t=1\ldots
  d}\in L^\infty(\omega, \R^{d^4})$ defined on an open, bounded, contractible
domain $\omega\subset\mathbb{R}^d$ with Lipschitz boundary, satisfy, for all $i,j,s,t,q=1\ldots d$:
\begin{equation}\label{symF}
\begin{split}
& F_{ij,st} = - F_{ji,st} = -F_{ij, ts}, \qquad F_{ij, st} = F_{st,ij},
\\ & F_{ij, st}+F_{is,tj} + F_{it,js} = 0, 
\\ & \partial_qF_{ij,st} + \partial_sF_{ij,tq} + \partial_tF_{ij,qs} = 0
\quad \mbox{ in the sense of distributions on }\; \omega. 
\end{split}
\end{equation}
Then, for any exponent $\alpha\in
(0,1)$, the set of $\mathcal{\mathcal{C}}^{1,\alpha}(\bar\omega,
\R^{\bar d})$-regular weak solutions to:
(\ref{MA}), is dense in $\mathcal{\mathcal{C}}^0(\bar\omega, \R^{\bar d})$. 
Namely, every $v\in \mathcal{\mathcal{C}}^0(\bar\omega,\R^{\bar d})$ is the
uniform limit of some sequence
$\{v_n\in\mathcal{\mathcal{C}}^{1,\alpha}(\bar\omega,\R^{\bar d})\}_{n=1}^\infty$, such that: 
$\mathfrak{Det}\,\nabla^2v= F$ on $\omega$, for all $n\geq 1$.
\end{corollary} 

\medskip

\subsection{Comparison with the literature} Our present result, in
addition to addressing the well-posedness of the geometrical
system (\ref{MA}), contributes to the growing body of work on the
flexibility of $\mathcal{C}^{1,\alpha}$ isometric immersions and
solutions to related PDEs. 

\medskip

\noindent The study of flexibility in the context of the system (\ref{II}) originates from the
classical Nash-Kuiper theorem \cite{Nash1, Kuiper} about
$\mathcal{C}^1$ isometric immersions of a $d$-dimensional Riemannian manifold
into $\mathbb{R}^{d+1}$ (in our terminology, this corresponds to
codimension $k=1$). The H\"older-regular isometric
immersions were studied by Borisov \cite{Borisov1958,Borisov1965}, who conjectured that
the Nash-Kuiper theorem extended to $\mathcal{C}^{1,\alpha}$-regularity for
$\alpha < (1 + d^2 + d)^{-1}$ and $\alpha < 1/5$ for $d = 2$.  
He provided a proof for the exponent $\alpha < 1/7$, dimension $d=2$
and the analytic metric $g$ in \cite{Borisov2004}. 
Borisov's conjecture was confirmed by Conti, De Lellis and Sz\'ekelyhidi
in \cite{CDS}, whereas the case
$\alpha < 1/5$, $d = 2$ was resolved by De Lellis, Inauen and Sz\'ekelyhidi in \cite{DIS1/5}.  
These results were generalized to compact manifolds by Cao and
Sz\'ekelyhidi in \cite{CaoSze2022}. Locally, it is very recently improved to 
$\alpha < (d^2 - d + 1)^{-1}$ by Cao, Hirsch and Inauen in \cite{CHI2}
and $\alpha<(2d-1)^{-1}$ by Inauen in \cite{In2026}.

\medskip

\noindent If the codimension $k$ is larger than $1$, one can construct more regular isometric
immersions. In this direction, Nash first showed \cite{Nash2} that any
manifold $d$-dimensional Riemannian manifold with metric $g \in
\mathcal{C}^r$, $r \geq 3$, admits a
$\mathcal{C}^r$-regular isometric immersion into $\mathbb{R}^{d+k}$
for sufficiently large $k$.  
The codimension bounds were later established and refined by Gromov
\cite{GromovPdr} and G\"unther \cite{Guenther}, while the case of  $g \in
\mathcal{C}^{r,\beta}$ with $r+\beta>2$, was was treated by Jacobowitz \cite{Jaco}.
When the metric's regularity is lower, i.e. $
r + \beta<2$, K\"all\'en \cite{Kallen} constructed a
$\mathcal{C}^{1,\alpha}$ immersion for any $\alpha < (r + \beta)/2$
when $k \geq 2d_*(d+1)$ is sufficiently large. 
See also \cite{DI2020,CaoIn2024,CaoSz2025} for further
results about isometric immersions in high codimension. 
Recently, Lewicka proved in  \cite{lew_full2d2k} that in dimension
$d=2$ and codimension $k\geq 2$, every
subsolution to (\ref{II}) with the metric $g\in\mathcal{C}^{r,\beta}$
admits an approximate sequence of isometric
immersions  with regularity $\mathcal{C}^{1,\alpha}$, for all
$\alpha<\min\{(r+\beta)/2, 1\}$. 
This property is referred to as {\em full flexibility}. In the present
paper, we extend this result to arbitrary dimension $d\geq 2$ and the corresponding codimension 
$\bar d$.

\medskip

\noindent The parallel version of the Nash-Kuiper scheme as in the papers
cited above, was first developed to handle 
flexibility of $\mathcal{C}^{1,\alpha}$ weak solutions to the 
Monge-Amp\`ere equation (\ref{MAeq}) by Lewicka and Pakzad in
\cite{lewpak_MA}, where the result was proved for
$\alpha<1/7$. This regularity exponent was later increased to
$\alpha<1/5$ and $\alpha<1/3$ by Cao and Szekelyhidi \cite{CS}, and
by Cao, Hirsch and Inauen \cite{CHI}, respectively. Closer to our present context,
in \cite{lew_conv} Lewicka introduced the Monge-Amp\`ere system
(\ref{MA}) and obtained flexibility of its solutions for any 
$\alpha<\min\{\beta/2, (1+2d_*/k)^{-1}\}$ given 
$A\in\mathcal{C}^{0,\beta}$, where the dimension $d$ and $k$ were arbitrary;
and also for any $\alpha<\min\{\beta/2, 1\}$ provided that $k\geq 2d_*$.
For the special case $d=2$, weak solutions of progressively higher
regularity, depending on the value of $k$, were obtained in the sequence of
papers \cite{lew_improved, lew_improved2, in_lew}. 
In \cite{in_lew2}, Inauen and Lewicka established the full flexibility
of the Monge-Amp\'ere system for $d=k=2$, serving as a direct precursor
to \cite{lew_conv} and to the present paper. 

\medskip

\noindent We also note that the convex integration technique has been
successfully applied to address Onsager’s conjecture for the Euler
equations, as well as to other fluid mechanics equations (see
\cite{BDSV, Isett, BV2021} and references therein). 

\medskip

\subsection{Overview of the strategy of proofs} \label{sub_over}

Our main technical contribution is the following {\em stage} construction.
whose iteration via the Nash-Kuiper algorithm in Theorem \ref{th_NK} ultimately provides the proof of Theorem \ref{th_final}. 
To ease the notation, we define the {\em defect}
relative to a matrix field $A$ and vector fields $v$ and $w$:
$$\mathcal{D}(A, v, w) \doteq A - \big(\frac{1}{2}(\nabla v)^T\nabla v +
\sym\nabla w\big).$$
 
\begin{theorem}\label{thm_stage} 
Let $\omega\subset\R^d$ be an open, bounded, $d$-dimensional set. Fix
two integers $N, K\geq 1$. There exists $\sigma_0\geq 1$ depending
only on $\omega, N, K$, such that the following holds.
Given $v\in\mathcal{C}^2(\bar\omega,\R^{\bar d})$, 
$w\in\mathcal{C}^2(\bar\omega,\R^{d})$, and 
$A\in\mathcal{C}^{r,\beta}(\bar\omega,\R^{d\times
  d}_\sym)$ with $r\in \{0,1\}$, $\beta\in (0,1]$, satisfying:
\begin{equation}\label{defect0}
0< \|\mathcal{D}\|_0\leq 1 \quad \mbox{where} \quad
\mathcal{D}=\mathcal{D}(A, v, w),
\end{equation}
and given two positive constants $\mathcal{M}, \sigma$ such that:
\begin{equation}\label{Assu}
\mathcal{M}\geq\max\{\|v\|_2, \|w\|_2, 1\},\quad \sigma\geq \sigma_0,
\end{equation}
there exist $\tilde v\in\mathcal{C}^2(\bar \omega,\R^{\bar d})$,
$\tilde w\in\mathcal{C}^{2}(\bar\omega, \R^d)$,
such that the following bounds are valid:
\begin{align*}
& \|\tilde v - v\|_1\leq C\|\mathcal{D}\|_0^{1/2}, \qquad\qquad
\|\tilde w -w\|_1\leq C\|\mathcal{D}\|_0^{1/2} \big(1+ \|\nabla v\|_0\big), 
\vspace{3mm} \tag*{(\theequation)$_1$}\refstepcounter{equation} \label{Abound12}\\
& \|\nabla^2\tilde v\|_0\leq C \mathcal{M}\sigma^{dK+N}, \hspace{12.5mm}
\|\nabla^2\tilde w\|_0\leq C \mathcal{M}\sigma^{dK+N} \big(1+\|\nabla v\|_0\big),  
\vspace{3mm} \tag*{(\theequation)$_2$}\label{Abound22} \\ 
& \displaystyle{\|\mathcal{D}(A, \tilde v, \tilde w)\|_0\leq
  C\Big(\frac{\|A\|_{r,\beta}}{\mathcal{M}^{r+\beta}}\|\mathcal{D}\|_0^{(r+\beta)/2}
  + \frac{\|\mathcal{D}\|_0}{\sigma^{KN}}\Big).}
\tag*{(\theequation)$_3$} \label{Abound32}
\end{align*}
Above, the constants $C$ depend only on $\omega, r,\beta$ and $N, K$.
\end{theorem}

\smallskip

\noindent We emphasize that, compared to the existing literature, we
are able to reduce the magnitude of the defect at the expense of a
smaller blow-up rate for the second derivatives of the subsolutions in
\ref{Abound22}–\ref{Abound32}. We now briefly outline how this is achieved.

\medskip

\noindent In the course of the proof, we inductively define $K$
intermediate subsolutions ${(v_k, w_k)}_{k=1}^K$. The initial pair
$(v_0, w_0)$ is constructed via a mollification of the given $(v,w)$,
while the final pair satisfies $(\tilde v, \tilde w) = (v_K, w_K)$. 
The construction of $(v_{k+1}, w_{k+1})$ from $(v_k, w_k)$ is carried
out by successively adding {\em Kuiper's corrugations}, as introduced
in Lemma \ref{lem_step2}, with the aim of canceling the $d_*$ rank-one
{\em primitive defects} of the form $a^2 \eta \otimes \eta$. These
primitive defects arise from the decomposition of 
$\mathcal{D}(A_0, v_k, w_k)$, see Lemma \ref{lem_dec_def}, where $A$
is replaced by its mollified version $A_0$, yielding the first term on
the right-hand side of the estimate \ref{Abound32}. 

\medskip

\noindent First, codimension $n=1$ is used to cancel the primitive
defects corresponding to $\eta = \eta_{1j}$ for all $j=1\ldots d$,
through a higher tier of intermediate subsolution pairs ${(V_{1j},
  W_{1j})}_{j=1}^d$. Here, Lemma \ref{lem_IBP} is applied $N$ times to
reduce the order of the resulting error terms. A prior application of
{\em K\"allen's iteration}, see Lemma \ref{lem_Kallen}, allows one to
additionally absorb into $\mathcal{D}$ those portions of the error
that cannot be handled by Lemma \ref{lem_IBP}. 
The subsequent codimensions $n=2 \ldots\bar d$ are 
used to consecutively cancel the remaining primitive defects, namely
those corresponding to $\eta = \eta_{ij}$ with $2 \leq i \leq j \leq d$, 
thereby constructing the next intermediate subsolution pairs ${(V_n, W_n)}_{n=2}^{\bar d}$.  
The main observation is that the first error term on the right-hand
side of (\ref{step_err}), which involves $\nabla^2 V_{n-1}^n$, can be
written as a sum of terms of the same structure, now involving only
$\nabla^2 v_0$, together with an expansion of oscillatory
contributions remaining from previous induction steps. All these terms oscillate
in a fixed direction and can therefore be further decomposed and
reduced by an argument similar to (\ref{ibp}) in Lemma \ref{lem_IBP}, but without the $P_i^j$
components. This observation constitutes the main new ingredient of our analysis. 
We then conclude the induction step by setting $(v_{k+1}, w_{k+1}) =(V_{\bar d}, W_{\bar d})$. 

\medskip

\noindent With Theorem \ref{thm_stage} at hand, Theorem \ref{th_final}
follows via a further iterative procedure, referred to as the {\em
  Nash-Kuiper algorithm}, see Theorem \ref{th_NK}. We briefly recall
its main idea, in which the solution $(\tilde v, \tilde w)$ to
(\ref{VK}) is obtained as the limit of the sequence ${(\tilde v_n,
  \tilde w_n)}_{n=1}^\infty$, generated through repeated applications
of Theorem \ref{thm_stage}. 
Along this iteration, the decay rate of $\|\tilde v_n - \tilde
v_{n-1}\|_1$ is roughly $\sigma^{-nKN/2}$, as follows from
\ref{Abound12} and \ref{Abound32}. Combining this with the blow-up
rate $\sigma^{n(dK+N)}$ of $\|\nabla^2 \tilde v_n\|_0$ in
\ref{Abound22}, and applying an interpolation inequality in H\"older
spaces, we obtain a bound on the corresponding ${\mathcal C}^{1,\alpha}$ norm: 
\begin{equation*}
\|\tilde v_n-\tilde v_{n-1}\|_{1,\alpha}\leq C\sigma^{n\big(\alpha(dK+N+KN/2)-KN/2\big)}.
\end{equation*}
For the sequence $\{\tilde v_n\}_{n=1}^\infty$ to be Cauchy in
${\mathcal C}^{1,\alpha}$, the power of the relative frequencies
$\sigma$ along the iteration must be negative, which is equivalent to:
\begin{equation*}
\alpha< \alpha_0\doteq \frac{KN/2}{dK+N+KN/2} = \frac{1}{1+2r_{K,N}} 
\qquad\mbox{where }\quad r_{K,N} = \frac{dK+N}{KN}. 
\end{equation*}
We observe that $r_{K,N}$ is the ratio of the blow-up rate of
$\|\nabla^2 \tilde v\|_0$ to the decay rate of $\|\mathcal{D}(A, \tilde
v, \tilde w)\|_0$ (specifically, the portion corresponding to the
second term on the right-hand side of \ref{Abound32}), and that it can be made arbitrarily small:
\begin{equation*}
r_{K,N} \to 0 \qquad \mbox{as } N, K\to\infty,
\end{equation*}
which corresponds to $\alpha_0 = 1$. Since the H\"older regularity
obtained by iterating Theorem \ref{thm_stage} is ultimately limited
only by the regularity of $A$ and by $\alpha_0$, this yields the range
stated in (\ref{VKrange}). 
\smallskip

\subsection{Organization of the paper and notation} 
In Section \ref{sec_step}, we collect several preparatory statements used in
the proof of Theorem \ref{thm_stage}. The proof, along with the stage
construction, is carried out in Section \ref{sec_khamsa}. The Nash-Kuiper
scheme argument, which yields the proof of Theorem \ref{th_final} is recalled in
section \ref{sec4}. Finally, Section \ref{sec_iso_imm} outlines
the additional arguments needed to extend Theorem \ref{th_final} to
the isometric immersion system (\ref{II}). 

\medskip

\noindent By $\mathbb{R}^{d\times d}_{\sym}$ we denote the space of symmetric
$d\times d$ matrices, and $\mathbb{R}^{d\times d}_{\sym, >}$ stands for the set of such matrices which are positive definite. The space of H\"older continuous vector fields
$\mathcal{C}^{m,\alpha}(\bar\omega,\R^k)$ consists of restrictions of
all $f\in \mathcal{C}^{m,\alpha}(\mathbb{R}^d,\R^k)$ to the closure of
an open, bounded set $\omega\subset\R^d$. The
$\mathcal{C}^m(\bar\omega,\R^k)$ norm of this restriction is
denoted by $\|f\|_m$, while its H\"older norm in $\mathcal{C}^{m,
  \alpha}(\bar\omega,\R^k)$ is $\|f\|_{m,\alpha}$. By $\nabla^{(m)}f$
we denote the table-valued field of all partial derivatives of $f$ of order
$m$; in the introduction section, we kept the notation $\nabla^2f$ for the matrix of second derivatives. If $d=k$, the symmetric part of $\nabla f$ is: $\sym\nabla f = (\nabla f + (\nabla f)^T)/2$.
By $C$ we denote a universal constant that may change from line to
line, but it depends only on the specified parameters.

\section{The preparatory statements}\label{sec_step}

In this section, we collect several technical ingredients that will be
used in the proof of Theorem \ref{thm_stage}. 
The first lemma concerns the
convolution and commutator estimates from \cite{CDS}.

\begin{lemma}\label{lem_stima}
Let $\phi\in\mathcal{C}_c^\infty(\R^d,\mathbb{R})$ be a standard
mollifier that is nonnegative, radially symmetric, supported on the
unit ball $B(0,1)\subset\R^d$ and such that $\int_{\mathbb{R}^d} \phi \dx = 1$. Denote: 
$$\phi_l (x) \doteq \frac{1}{l^d}\phi(\frac{x}{l})\quad\mbox{ for all
}\; l\in (0,1], \;  x\in\R^d.$$
Then, for every $f,g\in\mathcal{C}^0(\mathbb{R}^d,\R)$, every
$m\geq 0$ and $r\in\{0,1\}$, $\beta\in (0,1]$, there holds:
\begin{align*}
& \|\nabla^{(m)}(f\ast\phi_l)\|_{0} \leq
\frac{C}{l^m}\|f\|_0,\tag*{(\theequation)$_1$}\vspace{1mm} \refstepcounter{equation} \label{stima1}\\
& \|f - f\ast\phi_l\|_0\leq l^{r
+\beta} \|f\|_{r,\beta},\tag*{(\theequation)$_2$} \vspace{1mm} \label{stima2}\\
& \|\nabla^{(m)}\big((fg)\ast\phi_l - (f\ast\phi_l)
(g\ast\phi_l)\big)\|_0\leq {C}{l^{2- m}}\|\nabla f\|_{0}
\|\nabla g\|_{0}, \tag*{(\theequation)$_3$} \label{stima4}
\end{align*}
with constants $C>0$ depending only on $d$ and $m$. 
\end{lemma}

\medskip

\noindent The next lemma provides a decomposition of symmetric
matrices into linear combinations of rank-one {\em primitive matrices}.
The result is elementary and its proof is left to the reader. 

\begin{lemma}\label{lem_dec_def}
For all $i,j=1\ldots d$, define the unit vectors:
$$\eta_{ij}\doteq\frac{e_i+e_j}{|e_i+e_j|} \in\R^d.$$
Then $\{\eta_{ij}\otimes \eta_{ij}\}_{i,j=1\ldots d,\;i\leq j}$ is a basis of the
linear space $\R^{d\times d}_\sym$, so that:
$$H=\sum_{i\leq j}\bar a_{ij}(H) \eta_{ij}\otimes \eta_{ij}
\quad\mbox{ for all } H\in \R^{d\times d}_\sym,$$
where $\big\{\bar a_{ij}:\R^{d\times d}_\sym\to \R\big\}_{i,j=1\ldots d,\;i\leq j}$
are the linear coordinate functions. Moreover, denoting:
\begin{equation}\label{H_0}
H_0 \doteq \sum_{i\leq j}\eta_{ij}\otimes \eta_{ij},
\end{equation}
there exists $r_0\in (0, 1)$, depending only on $d$, such that:
$$|\bar a_{ij} (H)-1|\leq\frac{1}{2} \quad \mbox{ for all } H\in
B(H_0,r_0)\subset \R^{d\times d}_{\sym} \quad \mbox{and all }\; i \leq j.$$
\end{lemma}

\medskip

\noindent As the next ingredient, we recall the {\em step} construction
from \cite[Lemma 2.1]{lew_conv}, in which a single codimension is used
to cancel one rank-one defect of the form $a(x)^2e_i\otimes e_i$.

\begin{lemma}\label{lem_step2}
Let $v\in \mathcal{C}^1(\R^d, \R^{\bar d})$, $w\in \mathcal{C}^1(\R^d,
\R^{d})$, $\lambda>0$ and $a\in \mathcal{C}^2(\R^d,\R)$ be given. Denote: 
$$\Gamma(t) \doteq 2\sin t,\quad \bar\Gamma(t) \doteq -\frac{1}{2}\sin (2t),\quad
\dbar\Gamma(t) \doteq \frac{1}{2}\cos (2t),\quad 
\tbar\Gamma(t) \doteq 1- \frac{1}{2}\cos(2t).$$
Fix a coordinate $n=1\ldots \bar d$ and a unit vector $\eta\in\R^d$. Using the notation:
$$t_\eta \doteq \langle x, \eta\rangle,$$
let $\tilde v\in \mathcal{C}^1(\R^d, \R^{\bar d})$, $\tilde w\in
\mathcal{C}^1(\R^d, \R^{d})$ be given by:
\begin{equation}\label{defi_per}
\begin{split}
& \tilde v \doteq v + \frac{\Gamma(\lambda t_\eta)}{\lambda} a(x) e_n,
\\ & \tilde w \doteq w -\frac{\Gamma(\lambda t_\eta)}{\lambda} a(x)\nabla v^n
+ \frac{\Gamma(\lambda t_\eta)}{\lambda} a(x)^2\eta + 
\frac{\Gamma(\lambda t_\eta)}{\lambda^2} a(x)\nabla a(x).
\end{split}
\end{equation}
Then, the following identity is valid on $\R^d$:
\begin{equation}\label{step_err}
\begin{split}
& \big(\frac{1}{2}(\nabla \tilde v)^T \nabla \tilde v + \sym\nabla \tilde w\big) - 
\big(\frac{1}{2}(\nabla v)^T \nabla v + \sym\nabla w\big) - a(x)^2\eta\otimes \eta
\\ & = -\frac{\Gamma(\lambda t_\eta)}{\lambda} a \nabla^2 v^n
+ \frac{\dbar\Gamma(\lambda  t_\eta) }{\lambda^2} a\nabla^2 a
+ \frac{\tbar\Gamma(\lambda  t_\eta)}{\lambda^2}\nabla a\otimes\nabla a.
\end{split}
\end{equation}
\end{lemma}

\medskip

\noindent The proof of (\ref{step_err}) follows by direct inspection, in view of the identities:
$$ \frac{1}{2} (\Gamma')^2+ \bar\Gamma'=1, \qquad
\Gamma'\Gamma  + 2\bar\Gamma + \dbar\Gamma' =0, \qquad
\frac{1}{2}\Gamma^2 + \dbar\Gamma = \tbar\Gamma.$$
Next, we present the  {\em integration by parts} argument, put
forward in \cite{CHI2}. We remark that if $\Gamma_0$ below has the
general form $\alpha\sin(\beta t)$ or 
$\alpha \cos(\beta t)$, as is the case for the zero-mean periodic profiles $\Gamma$, $\bar\Gamma$
$\dbar\Gamma$ and $\tbar\Gamma -1$ in Lemma \ref{lem_step2}, then the
primitives $\Gamma_i$ defined in (\ref{recu_ibp}) are of the same form
$\frac{\alpha}{\beta^i}\sin(\beta t)$ or $\frac{\alpha}{\beta^i}\cos(\beta t)$. In particular, all their
derivatives are bounded, which ensures the uniformity of the estimates
in the proof of Theorem \ref{thm_stage}. 

\begin{lemma}\label{lem_IBP}
Given  $H\in\mathcal{C}^{k+1}(\R^d,\R^{d\times d}_\sym)$, $\lambda>0$,
$j=1\ldots d$ and $\Gamma_0\in\mathcal{C}(\R,\R)$, there holds:
\begin{equation}\label{ibp}
\begin{split}
\frac{\Gamma_0(\lambda t_{\eta_{1j}})}{\lambda}H = & \; (-1)^{k+1}\frac{\Gamma_{k+1}(\lambda
  t_{\eta_{1j}})}{\lambda^{k+2}} \sym\nabla L^{j}_k(H)  \\ & + \sym\nabla
\Big(\sum_{i=0}^k(-1)^i\frac{\Gamma_{i+1}(\lambda t_{\eta_{1j}})}{\lambda^{i+2}}L_i^j(H)\Big) 
+ \sum_{i=0}^{k}(-1)^i\frac{\Gamma_i(\lambda t_{\eta_{1j}})}{\lambda^{i+1}} P_i^j(H),
\end{split}
\end{equation}
where $t_{\eta_{ij}}=\langle x, \eta_{1j}\rangle$ with
$\eta_{1j} = \frac{e_1+e_j}{|e_1+e_j|}$. The functions $\Gamma_i\in
\mathcal{C}^{i}(\R,\R)$ satisfy the recursive definition: 
\begin{equation}\label{recu_ibp}
\Gamma_{i+1}' = \Gamma_{i}  \quad\mbox{ for all } \; i =0\ldots k,
\end{equation}
while $L_i^j(H)\in\mathcal{C}^{k+1-i}(\R^d,\R^d)$ and
$P_i^j(H)\in\mathcal{C}^{k+1-i}(\R^d,\R^{d\times d}_\sym)$ have the following properties:
\begin{itemize}
\item[(i)] the entries of $L_i^j(H)$ and $P_i^j(H)$ depend linearly on
  $H$ and consist of the linear combinations of the entries
  of $\nabla^{(i)}H$, 
\item[(ii)] $P_i^j(H)\in\mathrm{span}\big(\eta_{st}\otimes
  \eta_{st};\; 2\leq s\leq t\big)$, so that, recalling the
  decomposition in Lemma \ref{lem_dec_def}, there holds: $\bar
  a_{1t}(P_i^j(H))=0$ for all $t=1\ldots d$.
\end{itemize}
\end{lemma}
\begin{proof}
We assume that the matrix field $H$ is smooth and 
validate (\ref{ibp}) inductively for any $k\geq 0$, constructing the fields
$\{L_i^j =L_i^j(H)\}_{i\geq 0}$ and $\{P_i^j=P_i^j(H)\}_{i\geq 0}$ with the properties stated in (i), (ii). 
At $k=0$, we may check directly that these are valid in:
\begin{equation}\label{ibp0}
\begin{split}
& \frac{\Gamma_0(\lambda t_{\eta_{1j}})}{\lambda}H = \frac{\Gamma_0(\lambda t_{\eta_{1j}})}{\lambda}
\sym\big(L_0^{j}\otimes\eta_{ij}\big) + \frac{\Gamma_0(\lambda t_{\eta_{1j}})}{\lambda}P_0^j
\\ & \hspace{1.95cm} 
= -\frac{\Gamma_{1}(\lambda t_{\eta_{1j}})}{\lambda^{2}} \sym\nabla L^{j}_0  + \sym\nabla
\Big(\frac{\Gamma_{1}(\lambda t_{\eta_{1j}})}{\lambda^{2}}L_0^j\Big) 
+ \frac{\Gamma_0(\lambda t_{\eta_{1j}})}{\lambda} P_0^j, \\
& \mbox{where: } \;  L_0^j = 2\sqrt{2}He_1 - \sqrt{2}H_{11}(e_1+e_j),
\quad P_0^j = H-\sym(L_0^j\otimes \eta_{1j}).
\end{split}
\end{equation}
For the induction step, assume that $\{L_i^j, P_i^j\}_{i=0}^{k-1}$
are defined and (\ref{ibp}) holds at some $k-1\geq 0$:
\begin{equation}\label{ibpk-1}
\begin{split}
\frac{\Gamma_0(\lambda t_{\eta_{1j}})}{\lambda}H = & \; (-1)^{k}\frac{\Gamma_{k}(\lambda
  t_{\eta_{1j}})}{\lambda^{k+1}} \sym\nabla L^{j}_{k-1}  \\ & + \sym\nabla
\Big(\sum_{i=0}^{k-1}(-1)^i\frac{\Gamma_{i+1}(\lambda t_{\eta_{1j}})}{\lambda^{i+2}}L_i^j\Big) 
+ \sum_{i=0}^{k-1}(-1)^i\frac{\Gamma_i(\lambda t_{\eta_{1j}})}{\lambda^{i+1}} P_i^j.
\end{split}
\end{equation}
We now apply (\ref{ibp0}) to $H=\sym\nabla L_{k-1}^j$ and write:
\begin{equation*}
\begin{split}
& \frac{\Gamma_k(\lambda t_{\eta_{1j}})}{\lambda} \sym\nabla L_{k-1}^j
= -\frac{\Gamma_{k+1}(\lambda t_{\eta_{1j}})}{\lambda^{2}} \sym\nabla
L^{j}_k  + \sym\nabla
\Big(\frac{\Gamma_{k+1}(\lambda t_{\eta_{1j}})}{\lambda^{2}}L_k^j\Big) 
+ \frac{\Gamma_k(\lambda t_{\eta_{1j}})}{\lambda} P_k^j, \\
& \mbox{where: } \;  L_k^j = 2\sqrt{2}(\sym\nabla L_{k-1}^j)e_1 -
\sqrt{2}\langle \partial_1L_{k-1}^j, e_1\rangle (e_1+e_j), \\ & 
\hspace{1.35cm} P_k^j = \sym\nabla L_{k-1}^j - \sym(L_k^j\otimes \eta_{1j}).
\end{split}
\end{equation*}
Multiplying the above by $(-1)/\lambda^k$ and replacing with the
first term in the right hand side of (\ref{ibpk-1}), provides the
identity (\ref{ibp}). The recursive formulas for $L_k^j, P_k^j$ imply
the propagation of the properties (i), (ii). The proof is done. 
\end{proof}

\medskip

\noindent Finally, we recall a version of {\em K\"all\'en's
decomposition}, which originates from an iterative argument in 
\cite{Kallen}. In comparison with Lemma \ref{lem_dec_def}, the
decomposition below additionally absorbs the
non-oscillatory component  $\frac{1}{\lambda^2}\nabla a\otimes\nabla a$ 
of the last defect term on the right hand side of (\ref{step_err}).
The statement below is taken from \cite[Lemma 6.1]{lew_full2d2k} and
we point out that a similar result appeared in \cite[Proposition
3.1]{in_lew2} with only one extra term of the type
$\frac{1}{\lambda^2}\nabla a\otimes\nabla a$, while \cite[Lemma
2.2]{CHI2} featured more absorbed terms, as in (\ref{H_deco}).
The matrix field $H$ in (\ref{H_deco})
should be thought of as the scaled defect $\mathcal{D}$ in the proof
of Theorem \ref{thm_stage}. 

\begin{lemma}\label{lem_Kallen}
Fix two integers $M, N\geq 1$ and let $H\in \mathcal{C}^{M+N}(\R^d,\R^{d\times d}_\sym)$ satisfy:
\begin{equation}\label{ass_H0}
\|H-H_0\|_0\leq \frac{r_0}{2} \quad\mbox{ and }\quad 
\|\nabla^{(m)}H\|_0\leq \bar C\mu^m \quad \mbox{for all } \; m=1\ldots M+N,
\end{equation}
 for some given $\mu, \bar C>1$ and with $r_0$ as in Lemma \ref{lem_dec_def}.
Then, there exists $\sigma_0 \geq 1$ depending only on $d, M, N$,
such that the following holds. Let $\sigma\geq \sigma_0$ and define:
$$\lambda_{1j} = \sigma^j\mu \quad\mbox{ for } j=1\ldots d.$$
There exist $\{a_{ij}\in \mathcal{C}^{M+1}(\R^d,\R)\}_{i,j:1\ldots d, \, i\leq j}$
such that, writing:
\begin{equation}\label{H_deco}
H= \sum_{i\leq j} a_{ij}^2\eta_{ij}\otimes\eta_{ij} + \sum_{j=1}^d\frac{1}{\lambda_{1j}^2}\nabla
a_{1j}\otimes\nabla a_{1j} + \mathcal{F}_0,
\end{equation}
with $\{\eta_{ij}\in\R^d\}_{i,j:1\ldots d, \;i\leq j}$ as in Lemma \ref{lem_dec_def}, there hold the estimates:
\begin{equation}\label{Ebounds}
\begin{split}
&  \frac{1}{2}\leq a_{ij}^2\leq \frac{3}{2} \hspace{1.4cm} \mbox{ and }\quad 
\frac{1}{2}\leq a_{ij}\leq \frac{3}{2} \hspace{1.6cm} \mbox{ for all } \;  i\leq j,\\
& \|\nabla^{(m)} a_{ij}^2\|_0\leq C\mu^m\quad 
\mbox{and } \quad \|\nabla^{(m)} a_{ij}\|_0\leq C\mu^m \;\quad\mbox{ for }\;
 m=1\ldots M+1,\;  i\leq j,\\
& \|\nabla^{(m)}\mathcal{F}_0\|_0\leq C\frac{\mu^m}{\sigma^{2N}}\quad\mbox{ for } \; m=0\ldots M.  
\end{split}
\end{equation}
The constant $C$ above depends only on $d, M, N, \bar C$.
\end{lemma}

\smallskip

\noindent We note that if $H\in \mathcal{C}^\infty(\R^d,\R^{d\times d}_\sym)$, then
$\{a_{ij}\}_{i\leq j}$ and $\mathcal{F}$ are smooth as well, and their $(M+1)$-th and $M$-th
order derivatives, respectively, depend on the at most $M+N$ derivatives of $H$. 

\medskip

\section{The stage construction and a proof of Theorem \ref{thm_stage}} \label{sec_khamsa} 

We will now complete the {\em stage} in the present formulation of the 
convex integration algorithm. Recall that, given any
$v\in\mathcal{C}^1(\bar\omega,\R^{\bar d})$, 
$w\in\mathcal{C}^1(\bar\omega,\R^{d})$, 
$A\in\mathcal{C}^{r,\beta}(\bar\omega,\R^{d\times d}_\sym)$, we denote:
$$\mathcal{D}(A, v, w) \doteq A - \big(\frac{1}{2}(\nabla v)^T\nabla v +
\sym\nabla w\big).$$
As explained in the introduction section \ref{sec_intro}, the proof below inductively defines $K$
intermediate subsolutions $\{(v_k, w_k)\}_{k=1}^K$, over 13
steps. The initial pair $(v_0, w_0)$ is constructed in step 1, where we
also $A$ to $A_0$. The final pair, obtained in step 13, satisfies $(\tilde v, \tilde w) = (v_K, w_K)$. 

\medskip

\noindent The construction of $(v_{k+1}, w_{k+1})$ from $(v_k, w_k)$ is carried
out in steps 3-13. We successively add corrugations according to Lemma
\ref{lem_step2}, with the aim of canceling the
$d_*$ rank-one primitive defects $a^2 \eta \otimes
\eta$, see (\ref{step_err}). These primitive defects arise from the
decomposition of $\mathcal{D}(A_0, v_k, w_k)$ in Lemma \ref{lem_dec_def},
performed in step 3. In steps 4-8, we first use codimension $n=1$ to
cancel the primitive defects corresponding to all $\eta_{1j}$,
through a higher tier of intermediate subsolution pairs
${(V_{1j}, W_{1j})}_{j=1}^d$. Here, Lemma \ref{lem_IBP} is
applied $N$ times to reduce the order of the resulting error terms. A
prior application of Lemma \ref{lem_Kallen} allows one to additionally
absorb into $\mathcal{D}$ those portions of the error whose
oscillation profiles $\Gamma$ have nonzero mean and would otherwise be
of too large an order under the integration (\ref{recu_ibp}). 

\medskip

\noindent In steps 9-12, the codimensions $n=2 \ldots\bar d$ are
used to cancel the remaining primitive defects, namely those corresponding to $\eta_{ij}$ with $2 \leq i \leq j \leq
d$ (see the relation of $n$ and $(i,j)$ in (\ref{nij})), thereby
constructing the next intermediate subsolution pairs ${(V_n, W_n)}_{n=2}^{\bar d}$.   
The main observation is that the first error term on the right-hand
side of (\ref{step_err}), which involves $\nabla^{(2)} v_{k}^n$, can be
written as a sum of terms of the same structure, now involving only
$\nabla^{(2)} v_0$, together with an expansion of oscillatory
contributions remaining from previous induction steps. All these terms oscillate
in a fixed direction and can therefore be further decomposed and
reduced by an argument similar to (\ref{ibp}), but without the $P_i^j$
components. This observation enables the bounds established in step 12,
which constitute the main new ingredient of our analysis. 
We then conclude the induction step by setting $(v_{k+1}, w_{k+1}) =(V_{\bar d}, W_{\bar d})$. 


\begin{figure}[htbp]
\centering
\includegraphics[scale=0.53]{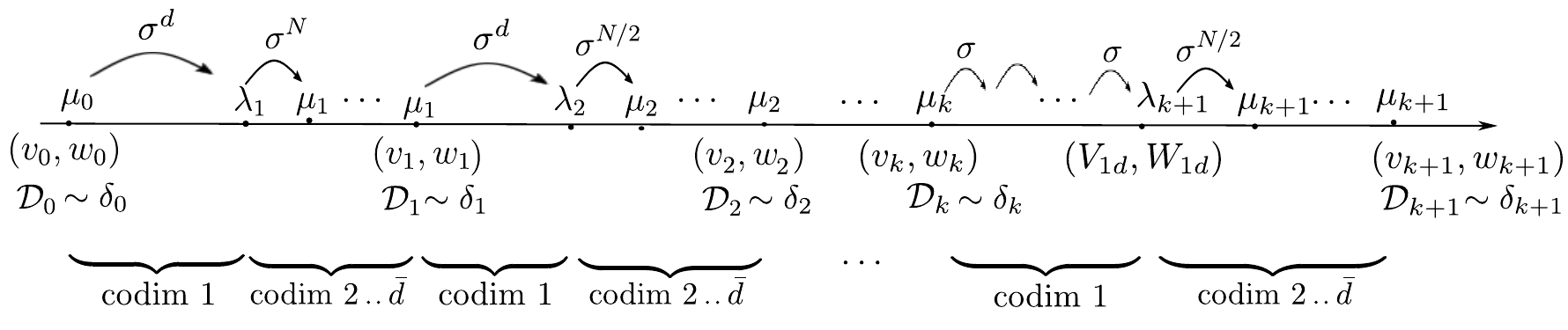}
\caption{{Progression of frequencies, intermediate fields, defects
 and codimensions used in the proof of Theorem \ref{thm_stage}.}} 
\label{fig_freq}
\end{figure}

\medskip

\noindent {\bf Proof of Theorem \ref{thm_stage}}. 
 
{\bf 1. (Mollification and the initial quantities)} 
Let $N$, $K$, $v$, $w$, $A$ be as
in the statement of the theorem, together with $\mathcal{M}$,
$\sigma$ satisfying (\ref{Assu}), where the relative frequency threshold $\sigma_0$ will
be determined in the course of the proof. Recall the definition of
$\mathcal{D}$ in (\ref{defect0}) and denote:
$$\delta_0 \doteq \|\mathcal{D}\|_0, \qquad \mu_0 \doteq \frac{\mathcal{M}}{\delta_0^{1/2}}.$$ 
We now define the mollified fields $v_0\in\mathcal{C}^\infty(\bar\omega, \R^{\bar d})$,
$w_0\in \mathcal{C}^\infty(\bar\omega, \R^{d})$,
$A_0\in\mathcal{C}^\infty(\bar\omega, \R^{d\times d}_\sym)$, relative
to the kernel $\phi_l$ as in Lemma \ref{lem_stima}, with $l=1/\mu_0$, namely:
$$v_0 \doteq v\ast \phi_{1/\mu_0},\quad w_0 \doteq w\ast
\phi_{1/\mu_0}, \quad A_0 \doteq A\ast \phi_{1/\mu_0}.$$
Denote the initial defect in the stage construction by:
$${\mathcal{D}}_0 \doteq \mathcal{D}(A_0, v_0, w_0) = A_0 -
\big(\frac{1}{2}(\nabla v_0)^T\nabla v_0 + \sym\nabla w_0\big).$$ 
We collect the following initial bounds:
\begin{align*}
& \|v_0-v\|_1 + \|w_0-w\|_1 \leq C \frac{\mathcal{M}}{\mu_0} = C\delta_0^{1/2},
\tag*{(\theequation)$_1$}\refstepcounter{equation} \label{pr_stima1}\\
& \|A_0-A\|_0 \leq C_A \frac{\|A\|_{r,\beta}}{\mu_0^{r+\beta}} = C_A
\frac{\|A\|_{r,\beta}}{\mathcal{M}^{r+\beta}} \delta_0^{(r+\beta)/2},
\tag*{(\theequation)$_2$} \label{pr_stima2}\\ 
& \|\nabla^{(m+2)}v_0\|_0 + \|\nabla^{(m+2)}w_0\|_0 \notag\\
&\qquad\quad \leq
C\mathcal{M}\mu_0^m = C\delta_0^{1/2} \mu_0^{m+1}
\quad \mbox{ for } m=0\ldots K(2N+4)+2,
\tag*{(\theequation)$_3$} \label{pr_stima3}\\ 
& \|\nabla^{(m)} \mathcal{D}_0\|_0\leq
C_0 \delta_0\mu_0^m \qquad \quad \mbox{ for }
m=0\ldots K(2N+4), \tag*{(\theequation)$_4$}\label{pr_stima4} 
\end{align*}
where the constants $C, C_0\geq 1$ depend only on $\omega, N, K$ and $C_A$
on $\omega, r,\beta$, as above. Indeed,  \ref{pr_stima1},
\ref{pr_stima2} result from \ref{stima2} in Lemma \ref{lem_stima}, 
while \ref{pr_stima3} follows by applying \ref{stima1} to 
$\nabla^{(2)}v$ and $\nabla^{(2)}w$ with the differentiability exponent $m-1$.
Since:
$$\mathcal{D}_0 = \mathcal{D}\ast \phi_{1/\mu_0} - \frac{1}{2}\big((\nabla
v_0)^T\nabla v_0 - ((\nabla v)^T\nabla v)\ast\phi_{1/\mu_0}\big), $$ 
we get \ref{pr_stima4} by applying \ref{stima1} to $\mathcal{D}$, and
\ref{stima4} to $\nabla v$.

\medskip

{\bf 2. (Setting up the induction)}
In steps 3-12, we will construct the intermediate fields 
$\{v_k\in\mathcal{C}^{(K-k)(2N+4)+2}(\bar\omega, \R^{\bar d})\}_{k=1}^K$ and
$\{w_k\in \mathcal{C}^{(K-k)(2N+4)+2}(\bar\omega, \R^{d})\}_{k=1}^K$, with their defects:
$$\mathcal{D}_k \doteq \mathcal{D}(A_0, v_k, w_k)   
= A_0 - \big(\frac{1}{2}(\nabla v_k)^T\nabla v_k + \sym\nabla w_k\big),$$
with respect to the progression of frequencies and the
defect magnitudes (see Figure \ref{fig_freq}) in:
\begin{equation}\label{DEF}
\begin{split}
& \lambda_k \doteq \mu_{k-1}\sigma^d, \qquad 
\mu_k \doteq \mu_0\sigma^{d+N+(d+\frac{N}{2})(k-1)},
\\ & \delta_k \doteq \frac{\delta_0}{\sigma^{Nk}}
\quad \mbox{ for all }\; k=1\ldots K.
\end{split}
\end{equation}
We will prove the following estimates, valid for all $k=1\ldots K$:
\begin{align*}
& \|v_k - v\|_1\leq C\delta_0^{1/2}, \qquad\qquad
\|w_k -w\|_1\leq C\delta_0^{1/2} \big(1+ \|\nabla v\|_0\big), 
\vspace{3mm} \tag*{(\theequation)$_1$}\refstepcounter{equation} \label{Bbound12}\\
& \|\nabla^{(m+2)} v_k^1\|_0\leq C\delta_{k-1}^{1/2}\lambda_k^{m+1} 
\qquad \mbox{for all } m=0\ldots (K-k)(2N+4),
\vspace{3mm} \tag*{(\theequation)$_2$}\label{Bbound22} \\ 
& \|\nabla^{(m+2)} v_k^n\|_0\leq C\delta_{k-1}^{1/2}\mu_k^{m+1}
\hspace{0.35cm} \quad \mbox{for all } n=2\ldots \bar d, \; m=0\ldots (K-k)(2N+4),
\vspace{3mm} \tag*{(\theequation)$_3$}\label{Bbound32} \\ 
& \|\nabla^{(m+2)} w_k\|_0\leq C\delta_{k-1}^{1/2}\mu_k^{m+1}\big(1+\|\nabla v\|_0\big)
\qquad \mbox{for all } m=0\ldots (K-k)(2N+4),
\vspace{3mm} \tag*{(\theequation)$_4$}\label{Bbound42} \\ 
& \|\nabla^{(m)}\mathcal{D}_k\|_0\leq C_k\delta_k\mu_k^m 
\hspace{3.45cm} \quad \mbox{for all } m=0\ldots (K-k)(2N+4).
\tag*{(\theequation)$_5$} \label{Bbound52}
\end{align*}
The fields $\{(v_k, w_k)\}_{k=1}^K$ will be defined by applying Lemma \ref{lem_step2}
with frequencies as in (\ref{DEF}), and with the corresponding amplitudes $\{a_{ij}^k\in
\mathcal{C}^{(K-k)(2N+4)+N+5}(\bar\omega, \R)\}_{i,j:1\ldots 
  d,\;i\leq j}$, $\{b_{ij}^k\in \mathcal{C}^{(K-k)(2N+4)+3}(\bar\omega, \R)\}_{i,j:2\ldots
  d,\;i\leq j}$ that obey the bound belows, valid for all $k=1\ldots K$:
\begin{align*}
& \|\nabla^{(m)}a_{ij}^k\|_0\leq C\delta_{k-1}^{1/2}\mu_{k-1}^m
\qquad \mbox{for } m=0\ldots (K-k)(2N+4) +N+5 \;\mbox{ and } i\leq j,
\tag*{(\theequation)$_1$} \refstepcounter{equation} \label{Cbound12} \\ 
& \|\nabla^{(m)}b_{ij}^k\|_0\leq C\delta_{k-1}^{1/2}\lambda_k^m
\hspace{1.05cm} \mbox{for } m=0\ldots  (K-k)(2N+4)+3 \;\mbox{ and } 2\leq i\leq j.
\tag*{(\theequation)$_2$}\label{Cbound22} 
\end{align*}
The constants $C, C_k\geq 1$ in \ref{Bbound12}-\ref{Bbound52} and
\ref{Cbound12}-\ref{Cbound22} will depend only on $\omega, N, K$.

\medskip

\noindent From now on, we fix $k=0\ldots K-1$. If $k\geq 1$,
we assume that \ref{Bbound12}-\ref{Bbound52} and
\ref{Cbound12}, \ref{Cbound22} hold for the counter values up to and
including $k$, in addition to the bounds
\ref{pr_stima1}-\ref{pr_stima4} at $k=0$. In the steps
3-12 below, we will validate \ref{Bbound12}-\ref{Bbound52} and
\ref{Cbound12}, \ref{Cbound22} at $k+1$. In the final step 13, we will 
deduce the statement of the theorem.

\medskip

{\bf 3. (The defect decomposition)} Define the intermediate
frequencies at the present $k$-counter value of the induction, noting that
$\lambda_{1d}$ equals $\lambda_{k+1}$ in  (\ref{DEF}):
$$\lambda_{1j} \doteq \mu_k\sigma^j \quad\mbox{ for all } j=1\ldots d.$$
We apply Lemma \ref{lem_Kallen} to the relative
frequency $\sigma$ and the following quantities: 
$$H= H_0+\frac{r_0}{2C_k\delta_k}\mathcal{D}_k, \quad \mu=\mu_k, \quad \sigma,$$
involving $H_0$ given in (\ref{H_0}). Consequently, 
we obtain the positive coefficients $a_{ij} = {\bar a_{ij}}^{1/2}$ in the
decomposition of $H$ above, together with the 
error field $\mathcal{F}_0$. Multiplying these by
$(2C_k\delta_k/r_0)^{1/2}$ and $2C_k\delta_k/r_0$, respectively, we obtain the scaled fields:
$$\{a_{ij}^{k+1}\in\mathcal{C}^{(K-k)(2N+4)-N+1}(\bar\omega, \R)\}_{i,j=1\ldots
  d,\; i\leq j}, \quad \mathcal{F}
\in \mathcal{C}^{(K-k)(2N+4)-N}(\bar\omega, \R^{d\times d}_\sym),$$ 
in the following decomposition of $\mathcal{D}_k$:
\begin{equation}\label{def1}
\mathcal{D}_k = \sum_{i\leq j} (a_{ij}^{k+1})^2\eta_{ij}\otimes
\eta_{ij}+\sum_{j=1}^d\frac{1}{\lambda_{1j}^2} \nabla a_{1j}^{k+1}\otimes \nabla a_{1j}^{k+1}
+\mathcal{F} - \frac{2C_k\delta_k}{r_0}H_0,
\end{equation}
that satisfy the scaled version of (\ref{Ebounds}):
\begin{equation}\label{Fbounds}
\begin{split}
&  \frac{C_k\delta_k}{r_0}\leq (a_{ij}^{k+1})^2\leq
\frac{3C_k\delta_k}{r_0} \quad\mbox{ and }\quad  
\frac{1}{2}\big(\frac{2C_k\delta_k}{r_0}\big)^{1/2}
\leq a_{ij}^{k+1}\leq \frac{3}{2} \big(\frac{2C_k\delta_k}{r_0}\big)^{1/2}
\quad  \mbox{for all } \;  i\leq j,\\
& \|\nabla^{(m)} (a_{ij}^{k+1})^2\|_0\leq C\delta_k\mu_k^m\quad 
\mbox{and } \quad \|\nabla^{(m)} a_{ij}^{k+1}\|_0\leq
C\delta_k^{1/2}\mu_k^m \\
& \hspace{4.55cm}\mbox{ for all }  m=1\ldots (K-k)(2N+4)-N+1,\;  i\leq j,\\
& \|\nabla^{(m)}\mathcal{F}\|_0\leq
C\delta_k\frac{\mu_k^m}{\sigma^{2N}}\hspace{9mm}\mbox{ for all } m=0\ldots (K-k)(2N+4)-N.  
\end{split}
\end{equation}
The constants $C$, $C_k$ depend only on $\omega, K, N$.
In particular, there follows \ref{Cbound12} at $k+1$.

\medskip

{\bf 4. (Adding corrugations in the first codimension)} 
We recursively define the fields: 
$\{V_{1j}\in\mathcal{C}^{(K-k)(2N+4)-N+1}(\bar\omega,\R^{\bar d})\}_{j=1}^d$,
$\{W_{1j}\in\mathcal{C}^{(K-k)(2N+4)-2N-1}(\bar\omega,\R^{d})\}_{j=1}^d$,  by setting:
$$(V_{1,0}, W_{1,0}) = (v_k, w_k)$$
and further, utilizing construction (\ref{defi_per}) in Lemma \ref{lem_step2} for each $j=1\ldots d$,
with $a=a_{ij}$ (we suppress the superscript $k+1$ to ease the
notation), $\lambda = \lambda_{1j}$, $\eta=\eta_{1j}$ and $n=1$:
\begin{equation}\label{VW1}
\begin{split}
& V_{1j}\doteq V_{1,j-1}+\frac{a_{1j}}{\lambda_{ij}}\Gamma(\lambda_{1j}t_{\eta_{1j}})e_1,
\\ & W_{1j} \doteq W_{1,j-1} - \frac{a_{1j}}{\lambda_{1j}} \Gamma(\lambda_{1j} t_{\eta_{1j}})\nabla V^1_{1,j-1}
+ \frac{(a_{1j})^2}{\lambda_{1j}} \bar\Gamma(\lambda_{1j}
t_{\eta_{1j}})\eta_{1j} \\ &\qquad \quad+ 
\frac{a_{1j}}{(\lambda_{1j})^2} \dbar \Gamma(\lambda_{1j}
t_{\eta_{1j}})\nabla a_{1j} - \bar\Theta_{1j}.
\end{split}
\end{equation}
Note the additional perturbation term in $W_{1j}$, that is given by:
\begin{equation}\label{omega1}
\begin{split}
& \bar\Theta_{11} \doteq \Theta_{11} + \frac{2C_k\delta_k}{r_0}H_0x, \qquad 
\bar\Theta_{1j} \doteq \Theta_{1j} \quad\mbox{ for } j=2\ldots d,
\\ & \Theta_{1j} = -\sum_{i=0}^N (-1)^i \frac{\Gamma_{i+1}(\lambda_{1j}t_{\eta_{1j}})}{(\lambda_{1j})^{i+2}} 
L_i^j(a_{1j}\nabla^{(2)}V_{1,{j-1}}^1) \\ & \qquad\quad +  \sum_{i=0}^N (-1)^i
\frac{\dbar\Gamma_{i+1}(\lambda_{1j}t_{\eta_{1j}})}{(\lambda_{1j})^{i+3}}  
L_i^j(a_{1j}\nabla^{(2)} a_{1j}) \\ & \qquad\quad + \sum_{i=0}^N (-1)^i
\frac{(\tbar\Gamma-1)_{i+1}(\lambda_{1j}t_{\eta_{1j}})}{(\lambda_{1j})^{i+3}}  
L_i^j(\nabla a_{1j}\otimes \nabla a_{1j}),
\end{split}
\end{equation}
where $\{L_i^j\}_{i=0}^N$ are the operators defined in Lemma \ref{lem_IBP}.
By (\ref{ibp}) and (\ref{step_err}), we get:
\begin{equation*}
\begin{split}
& \mathcal{D}(A_0, V_{1j}, W_{1j}) - \mathcal{D}(A_0, V_{1,j-1},
W_{1,j-1}) \\ & = -\Big(\big(\frac{1}{2}(\nabla V_{1j})^T \nabla V_{1j} +
\sym\nabla W_{1j}\big) - \big(\frac{1}{2}(\nabla V_{1,j-1})^T \nabla V_{1,j-1} +
\sym\nabla W_{1,j-1}\big)\Big)\\ & = - (a_{1j})^2\eta_{1j}\otimes\eta_{1j} + 
\frac{\Gamma(\lambda_{1j}t_{\eta_{1j}})}{\lambda_{1j}} a_{1j}\nabla^{(2)}V_{1,{j-1}}^1 
- \frac{\dbar\Gamma(\lambda_{1j}t_{\eta_{1j}})}{(\lambda_{1j})^{2}}  
a_{1j}\nabla^{(2)} a_{1j} \\ & \qquad - \frac{(\tbar\Gamma-1)(\lambda_{1j}t_{\eta_{1j}})}{(\lambda_{1j})^{2}}  
\nabla a_{1j}\otimes \nabla a_{1j} - \frac{1}{(\lambda_{1j})^{2}}  
\nabla a_{1j}\otimes \nabla a_{1j} +\sym\nabla\bar\Theta_{1j}.
\end{split}
\end{equation*}
We above formula can be rewritten as:
\begin{equation}\label{def2}
\begin{split}
& \mathcal{D}(A_0, V_{1j}, W_{1j}) - \mathcal{D}(A_0, V_{1,j-1},
W_{1,j-1}) \\ & = \Big( - (a_{1j})^2\eta_{1j}\otimes\eta_{1j} -
\frac{1}{(\lambda_{1j})^{2}}  
\nabla a_{1j}\otimes \nabla a_{1j} \Big) - \mathcal{F}_{1j} - G_{1j} +
\left\{\begin{array}{ll} \displaystyle{\frac{2C_k\delta_k}{r_0}H_0} &\mbox{if }
    j=1\\ 0& \mbox{if } j\geq 2,\end{array}\right.
\end{split}
\end{equation}
where the error fields $\{\mathcal{F}_{1j}\}_{j=1}^d$ collect the following terms:
\begin{equation}\label{F1j}
\begin{split}
& \mathcal{F}_{1j}\doteq  (-1)^{N+1} \frac{\Gamma_{N+1}(\lambda_{1j}t_{\eta_{1j}})}{(\lambda_{1j})^{N+2}} 
\sym\nabla L_N^j(-a_{1j}\nabla^{(2)}V_{1,{j-1}}^1) \\ & \qquad\quad + (-1)^{N+1}
\frac{\dbar\Gamma_{N+1}(\lambda_{1j}t_{\eta_{1j}})}{(\lambda_{1j})^{N+3}}  
\sym\nabla L_N^j(a_{1j}\nabla^{(2)} a_{1j}) \\ & \qquad\quad + (-1)^{N+1}
\frac{(\tbar\Gamma-1)_{N+1}(\lambda_{1j}t_{\eta_{1j}})}{(\lambda_{1j})^{N+3}}  
\sym\nabla L_N^j(\nabla a_{1j}\otimes \nabla a_{1j}),
\end{split}
\end{equation}
while $\{G_{1j}\}_{j=1}^d$ are given by means of the operators
$\{P_i^j\}_{i=0}^N$ in Lemma \ref{lem_IBP}:
\begin{equation}\label{G1j}
\begin{split}
& G_{1j}= \sum_{i=0}^N (-1)^i \frac{\Gamma_{i}(\lambda_{1j}t_{\eta_{1j}})}{(\lambda_{1j})^{i+1}} 
P_i^j(-a_{1j}\nabla^{(2)}V_{1,{j-1}}^1) \\ & \qquad\quad +  \sum_{i=0}^N (-1)^i
\frac{\dbar\Gamma_{i}(\lambda_{1j}t_{\eta_{1j}})}{(\lambda_{1j})^{i+2}}  
P_i^j(a_{1j}\nabla^{(2)} a_{1j}) \\ & \qquad\quad + \sum_{i=0}^N (-1)^i
\frac{(\tbar\Gamma-1)_{i}(\lambda_{1j}t_{\eta_{1j}})}{(\lambda_{1j})^{i+2}}  
P_i^j(\nabla a_{1j}\otimes \nabla a_{1j}).
\end{split}
\end{equation}
Recall that there holds:
\begin{equation}\label{higher}
G_{1j}(x)\in\mbox{span}\{\eta_{st}\otimes \eta_{st};
2\leq s\leq t\} \quad \mbox{ for all } j=1\ldots d, \quad x\in\bar\omega.
\end{equation}
Inserting (\ref{def2}) into (\ref{def1}), we observe that:
\begin{equation}\label{def3}
\begin{split}
& \mathcal{D}(A_0, V_{1d}, W_{1d}) = \mathcal{D}_k + \sum_{j=1}^d
\Big(\mathcal{D}(A_0, V_{1j}, W_{1j}) - \mathcal{D}(A_0, V_{1,j-1}, W_{1,j-1})\Big)
\\ & = \sum_{2\leq i\leq j} (a_{ij}^{k+1})^2\eta_{ij}\otimes\eta_{ij}
+\Big(\mathcal{F} - \sum_{j=1}^d\mathcal{F}_{1j} \Big) - \sum_{j=1}^d G_{1j}. 
\end{split}
\end{equation}

\medskip

{\bf 5. (Bounds on $V_{1j}$)} Recalling the definition (\ref{VW1}),
the newly constructed fields $V_{1j}$ satisfy:
$$V_{1j}^1 = v_k + \sum_{i=1}^j (V_{1i}^1-V_{1,i-1}^1).$$
Further, by (\ref{Fbounds}) there holds, for all $m=0\ldots (K-k)(2N+4)-N+1$:
\begin{equation*}
\begin{split}
\|\nabla^{(m)}(V_{1j}^1-V_{1,j-1}^1)\|_0 & \leq C\sum_{p+q=m}
(\lambda_{1j})^{p-1}\|\nabla^{(q)}a_{1j} \|_0\\ &\leq C\sum_{p+q=m}
(\lambda_{1j})^{p-1}\delta_{k}^{1/2}\mu_k^q\leq
C\delta_k^{1/2}(\lambda_{1j})^{m-1}.
\end{split}
\end{equation*}
Applying \ref{pr_stima1} when $k=0$ and \ref{Bbound12} if $k\geq 1$,
it consequently transpires that:
\begin{equation}\label{V1j-1}
\|V_{1j}^1- v^1\|_1\leq \|v_k-v\|_1 + \sum_{i=1}^j \|V_{1i}^1-V_{1,i-1}^1\|_1
\leq C\delta_0^{1/2}+  C\delta_k^{1/2}\leq C\delta_0^{1/2}, 
\end{equation}
while for all $m=0\ldots (K-k)(2N+4)-N+1$ we get:
\begin{equation}\label{V1j-m}
\begin{split}
\|\nabla^{(m+2)}V_{1j}^1\|_0& \leq \|\nabla^{(m+2)}v_k^1\|_0 +
\sum_{i=1}^j \|\nabla^{(m+2)}(V_{1i}^1-V_{1,i-1}^1)\|_0
\\ & \leq  \|\nabla^{(m+2)}v_k^1\|_0 +
C \delta_k^{1/2}(\lambda_{1j})^{m+1}\leq C \delta_k^{1/2}(\lambda_{1j})^{m+1}.
\end{split}
\end{equation}
The final inequality above follows by noting that
$\{\delta_i^{1/2}\mu_i\}_{i=0}^K$ is an increasing sequence:
\begin{equation}\label{increase}
\frac{\delta_{i+1}^{1/2}\mu_{i+1}}{\delta_i^{1/2}\mu_i} =
\frac{1}{\sigma^{N/2}} \frac{\mu_{i+1}}{\mu_i} \geq
\frac{\sigma^{d+\frac{N}{2}}}{\sigma^{N/2}} \geq 1 \quad \mbox{ for } i=0\ldots K-1.
\end{equation}
In conclusion, we get (\ref{V1j-m}) at $k=0$ by \ref{pr_stima3}, because:
$$\|\nabla^{(m+2)}v_0\|_0\leq C \delta_0^{1/2} \mu_0^{m+1}\leq C
(\delta_0^{1/2} \mu_0)(\lambda_{1j})^m \leq C (\delta_k^{1/2}
\mu_k)(\lambda_{1j})^m \leq C \delta_0^{1/2} (\lambda_{1j})^{m+1},$$
while at $k\geq 1$, we use the inductive assumption \ref{Bbound22} in:
\begin{equation*}
\begin{split}
\|\nabla^{(m+2)}v_k^1\|_0& \leq C \delta_{k-1}^{1/2} (\lambda_k)^{m+1}\leq C
(\delta_{k-1}^{1/2} \lambda_k)(\lambda_{1j})^m = C (\delta_k^{1/2}\sigma^{N/2}
\lambda_k)(\lambda_{1j})^m \\ & \leq C (\delta_k^{1/2}
\mu_k)(\lambda_{1j})^m \leq C \delta_0^{1/2} (\lambda_{1j})^{m+1}.
\end{split}
\end{equation*}
In particular, since $v_{k+1}$ will be defined to differ from
$V_{1d}$ only on its $n=2\ldots \bar d$ coordinates, the bounds (\ref{V1j-1}),
(\ref{V1j-m}) readily imply a partial bound in \ref{Bbound12}
and \ref{Bbound22} at $k+1$, since:
\begin{equation}\label{fin1}
\begin{split}
& \|v_{k+1}^1-v^1\|_1\leq C\delta_0^{1/2}, \\
& \|\nabla^{(m+2)}v_{k+1}^1\|_0 \leq C \delta_k^{1/2}
\lambda_{k+1}^{m+1} \quad \mbox{ for all } m=0\ldots (K-k)(2N+4)-N+1.
\end{split}
\end{equation}

\medskip

{\bf 6. (Bounds on $W_{1j}$)} Similarly as for $V_{1j}$, we write:
$$W_{1j} = w_k + \sum_{i=1}^j (W_{1i}-W_{1,i-1}),$$
and, recalling the definition (\ref{VW1}), prepare the estimate:
\begin{equation}\label{W1j-diff}
\begin{split}
\|\nabla^{(m)}(W_{1j}-W_{1,j-1})\|_0  \leq & \; C\sum_{p+q+r=m}
(\lambda_{1j})^{p-1}\|\nabla^{(q)}a_{1j} \|_0\|\nabla^{(r+1)}V_{1, j-1}^1\|_0
\\ & + C\sum_{p+q=m} (\lambda_{1j})^{p-1}\|\nabla^{(q)}(a_{1j})^2 \|_0
\\ & + C\sum_{p+q+r=m}
(\lambda_{1j})^{p-2}\|\nabla^{(q)}a_{1j} \|_0 \|\nabla^{(r+1)}a_{1j}
\|_0 + \|\nabla^{(m)}\bar\Theta_{1j}\|_0
\end{split}
\end{equation}
To carry out estimating the four terms in the right hand side above, we first use
(\ref{V1j-m}), \ref{Bbound22}, \ref{pr_stima3} to check that for
$m=0\ldots (K-k)(2N+4)-N-1$:
\begin{equation}\label{Vj-1}
\|\nabla^{(m+2)}V_{1, j-1}^1\|_0\leq C \left\{\begin{array}{ll}
    \delta_k^{1/2}(\lambda_{1,j-1})^{m+1} & \mbox{if } j\geq 2, \; k\geq 0\\
\delta_{k-1}^{1/2} (\lambda_k)^{m+1} & \mbox{if } j=1, \; k\geq 1\\
\delta_0^{1/2}\mu_0^{m+1} & \mbox{if } j\geq 1, \; k=0\\
\end{array}\right\} \leq C\delta_k^{1/2}(\lambda_{1,j-1})^{m+1},
\end{equation}
where the final inequality above is due, as before, to
(\ref{increase}), and with the convention that $\lambda_{1,0}\doteq\mu_k$. 
Consequently, in virtue of (\ref{Fbounds}) and  (\ref{V1j-1}), we get
that for all $m=0\ldots (K-k)(2N+4)-N$:
\begin{equation*}
\begin{split}
& \sum_{p+q+r=m} (\lambda_{1j})^{p-1}\|\nabla^{(q)}a_{1j} \|_0\|\nabla^{(r+1)}V_{1, j-1}^1\|_0
\\ & \qquad \leq C \sum_{p+q=m}
(\lambda_{1j})^{p-1}\delta_k^{1/2}\mu_k^q\big(\delta_0^{1/2}+\|\nabla v\|_0\big)
+ C\hspace{-3mm} 
\sum_{p+q+r=m,\; r\geq 1}  (\lambda_{1j})^{p-1}\delta_k^{1/2}\mu_k^q\delta_k(\lambda_{1,j-1})^{r}
\\ & \qquad \leq C\delta_k^{1/2}(\lambda_{1j})^{m-1}\big(1+ \|\nabla v\|_0\big),\vspace{2mm}\\
& \sum_{p+q=m} (\lambda_{1j})^{p-1}\|\nabla^{(q)}(a_{1j})^2 \|_0
+ \hspace{-2mm} \sum_{p+q+r=m}
(\lambda_{1j})^{p-2}\|\nabla^{(q)}a_{1j} \|_0 \|\nabla^{(r+1)}a_{1j} \|_0 \\ & \qquad \leq 
C\sum_{p+q=m} (\lambda_{1j})^{p-1}\delta_k\mu_k^q 
+ C\sum_{p+q+r=m} (\lambda_{1j})^{p-2} \delta_k\mu_k^{q+r+1}\leq C
\delta_k(\lambda_{1j})^{m-1}.
\end{split}
\end{equation*}
Finally, recalling the definition (\ref{omega1}) and Lemma
\ref{lem_IBP}, we note that:
\begin{equation*}
\begin{split}
\|\nabla^{(m)}\Theta_{1j}\|_0\leq C\sum_{i=0}^N
\Big(& \sum_{p+q=m}(\lambda_{1j})^{p-(i+2)}\|\nabla^{(q+i)}(a_{1j} \nabla^{(2)}V_{1, j-1}^1)\|_0
\\ & + \sum_{p+q=m}(\lambda_{1j})^{p-(i+3)}\|\nabla^{(q+i)}(a_{1j} \nabla^{(2)}a_{1, j})\|_0
\\ & + \sum_{p+q=m}(\lambda_{1j})^{p-(i+3)}\|\nabla^{(q+i)}(\nabla a_{1j} \otimes \nabla a_{1j})\|_0\Big),
\end{split}
\end{equation*}
which yields for all $m=0\ldots (K-k)(2N+4)-2N-1$:
\begin{equation*}
\begin{split}
& \|\nabla^{(m)}\Theta_{1j}\|_0 
\leq C\sum_{i=0}^N \sum_{p+q=m} \Big(\sum_{r+t=q+i}(\lambda_{1j})^{p-(i+2)}
\|\nabla^{(r)}a_{1j}\|_0\| \nabla^{(t+2)}V_{1, j-1}^1\|_0 
\\ & \hspace{4.7cm} + \sum_{r+t=q+i} (\lambda_{1j})^{p-(i+3)}\|\nabla^{(r)}a_{1j}\|_0\| \nabla^{(t+2)}a_{1j}\|_0
\\ & \hspace{4.7cm} + \sum_{r+t=q+i}(\lambda_{1j})^{p-(i+3)}\|\nabla^{(r+1)}a_{1j} \|_0\|
\nabla^{(t+1)} a_{1j}\|_0 \Big)  
\\ & \leq C\sum_{i=0}^N \sum_{p+q=m} \Big(\sum_{r+t=q+i} (\lambda_{1j})^{p-(i+2)}\delta_k^{1/2}\mu_k^r
\delta_k^{1/2}(\lambda_{1,j-1})^{t+1} + \sum_{r+t=q+i} \hspace{-2mm} 
(\lambda_{1j})^{p-(i+2)}\delta_k\mu_k^{r+t+2}\Big)
\\ &\leq C\sum_{i=0}^N \sum_{p+q=m} \Big( (\lambda_{1j})^{p-(i+2)}\delta_k(\lambda_{1,j-1})^{q+i+1}
+ (\lambda_{1j})^{p-(i+2)}\delta_k\mu_k^{q+i+2}\Big) \\ & 
\leq C \delta_k (\lambda_{1j})^{m-1}\sum_{i=0}^N
\Big( \big(\frac{\lambda_{1,j-1}}{\lambda_{1j}}\big)^{i+1}
+ \big(\frac{\mu_k}{\lambda_{1j}}\big)^{i+2}\Big) \leq C \delta_k(\lambda_{1j})^{m-1}.
\end{split}
\end{equation*}
In conclusion, the bound (\ref{W1j-diff}) becomes:
\begin{equation*}
\begin{split}
\|\nabla^{(m)}(W_{1j}-W_{1,j-1})\|_0  \leq
C\delta_k^{1/2}(\lambda_{1j})^{m-1}&\big(1+ \|\nabla v\|_0\big) +
C\delta_k \|\nabla^{(m)}(H_0x)\|_0  \\ & \mbox{for all } m =0\ldots (K-k)(2N+4)-2N-1.
\end{split}
\end{equation*}
Therefore, my means of \ref{pr_stima1}, \ref{pr_stima3}, \ref{Bbound12}, \ref{Bbound42}:
\begin{equation}\label{fin2}
\begin{split}
& \|W_{1d}-w\|_1\leq \|w_k-w\|_1 +
\sum_{j=1}^d\|W_{1j}-W_{1,j-1}\|_1\leq C\delta_0^{1/2} \big(1+ \|\nabla v\|_0\big), \\
& \|\nabla^{(m+2)}W_{1d}\|_0 \leq C \|\nabla^{(m+2)}w_k\|_0 +
\sum_{j=1}^d \|\nabla^{(m+2)}(W_{1j}-W_{1,j-1})\|_0  \\ & \qquad
\qquad\quad 
\leq  C\left\{\begin{array}{ll}\delta_{k-1}^{1/2}\mu_k^{m+1}\big(1+ \|\nabla
    v\|_0\big) & \mbox{if } k\geq 1 \\ \delta_0^{1/2}\mu_0^{m+1} &
    \mbox{if } k=0\end{array}\right\} + 
C\delta_{k}^{1/2}\lambda_k^{m+1}\big(1+ \|\nabla v\|_0\big)
\\ & \qquad \qquad\quad \leq C \delta_{k}^{1/2}\mu_{k+1}^{m+1}\big(1+ \|\nabla v\|_0\big)
\quad \mbox{ for all } m=0\ldots (K-k)(2N+4)-2N-3,
\end{split}
\end{equation}
where, similarly to (\ref{increase}), we have noted that the sequence
$\{\delta_i^{1/2}\mu_{i+1}\}_{i=0}^{K-1}$ is increasing:
\begin{equation}\label{increase2}
\frac{\delta_{i+1}^{1/2}\mu_{i+2}}{\delta_i^{1/2}\mu_{i+1}} =
\frac{1}{\sigma^{N/2}} \frac{\mu_{i+2}}{\mu_{i+1}} =
\frac{\sigma^{d+\frac{N}{2}}}{\sigma^{N/2}} \geq 1 \quad \mbox{ for } i=0\ldots K-2.
\end{equation}

\medskip

{\bf 7. (Bounds on $\mathcal{F}_{1j}$)} Recalling the definition
(\ref{F1j}) and Lemma \ref{lem_IBP}(i), we proceed to bound:
\begin{equation*}
\begin{split}
\|\nabla^{(m)}\mathcal{F}_{1j}\|_0 \leq C\sum_{p+q=m}\Big(& \sum_{r+t=q+N+1}
(\lambda_{1j})^{p-(N+2)} \|\nabla^{(r)}a_{1j}\|_0\|\nabla^{(t+2)}V_{1,{j-1}}^1\|_0  \\ &
+ \sum_{r+t=q+N+1}
(\lambda_{1j})^{p-(N+3)}  \|\nabla^{(r)}a_{1j}\|_0\|\nabla^{(t+2)} a_{1j}\|_0 \\ & 
+ \sum_{r+t=q+N+1} (\lambda_{1j})^{p-(N+3)}  
\|\nabla^{(r+1)} a_{1j}\|_0\| \nabla^{(t+1)} a_{1j}\|_0 \Big).
\end{split}
\end{equation*}
From (\ref{Fbounds}) and (\ref{Vj-1}), it hence follows that for all $m=0\ldots (K-k)(2N+4)-2N-2$:
\begin{equation}\label{F1j_bound}
\begin{split}
\|\nabla^{(m)}\mathcal{F}_{1j}\|_0 & \leq  C\sum_{p+q=m}\Big(\sum_{r+t=q+N+1}
(\lambda_{1j})^{p-(N+2)} \delta_k^{1/2}\mu_k^r\delta_k^{1/2}
(\lambda_{1,j-1})^{t+1}\\ & \hspace{2cm} + \sum_{r+t=q+N+1}
(\lambda_{1j})^{p-(N+3)} \delta_k\mu_k^{r+t+2} \Big) \\ & \leq
C\delta_k \sum_{p+q=m}\Big( (\lambda_{1j})^{p-(N+2)}(\lambda_{1,j-1})^{q+N+2}
+ (\lambda_{1j})^{p-(N+3)}(\lambda_{1,j-1})^{q+N+3}\Big) \\ &
= C\delta_k (\lambda_{1j})^m\Big(\big(\frac{\lambda_{1,j-1}}{\lambda_{1j}}\big)^{N+2}
+ \big(\frac{\lambda_{1,j-1}}{\lambda_{1j}}\big)^{N+3}\Big)\leq
C\frac{\delta_k}{\sigma^{N+2}}\lambda_{k+1}^{m}.
\end{split}
\end{equation}

\medskip

{\bf 8. (Bounds on $G_{1j}$ and the updated decomposition coefficients $b_{ij}$)} 
As in the previous step, we use Lemma \ref{lem_IBP}(i) to bound the terms in (\ref{G1j}):
\begin{equation*}
\begin{split}
\|\nabla^{(m)} G_{1j}\|_0 \leq C \sum_{i=0}^N \sum_{p+q=m}\Big(&\sum_{r+t=q+i}
(\lambda_{1j})^{p-(i+1)} \|\nabla^{(r)}a_{1j}\|_0\|\nabla^{(t+2)}V_{1,{j-1}}^1\|_0  \\ &
+ \sum_{r+t=q+i}
(\lambda_{1j})^{p-(i+2)}  \|\nabla^{(r)}a_{1j}\|_0\|\nabla^{(t+2)} a_{1j}\|_0 \\ & 
+ \sum_{r+t=q+i} (\lambda_{1j})^{p-(i+2)}  
\|\nabla^{(r+1)} a_{1j}\|_0\| \nabla^{(t+1)} a_{1j}\|_0 \Big).
\end{split}
\end{equation*}
From (\ref{Fbounds}) and (\ref{Vj-1}) we get that, for all $m=0\ldots (K-k)(2N+4)-2N-1$:
\begin{equation*}
\begin{split}
\|\nabla^{(m)} G_{1j}\|_0 & \leq C \sum_{i=0}^N \sum_{p+q=m}\Big(\sum_{r+t=q+i}
(\lambda_{1j})^{p-(i+1)} \delta_k^{1/2}\mu_k^r\delta_k^{1/2}
(\lambda_{1,j-1})^{t+1}\\ & \hspace{3cm} + \sum_{r+t=q+i}
(\lambda_{1j})^{p-(i+2)} \delta_k\mu_k^{r+t+2} \Big) \\ & \leq
C\delta_k \sum_{i=0}^N \sum_{p+q=m}\Big( (\lambda_{1j})^{p-(i+1)}(\lambda_{1,j-1})^{q+i+1}
+ (\lambda_{1j})^{p-(i+2)}(\lambda_{1,j-1})^{q+i+2}\Big) \\ &
= C\delta_k (\lambda_{1j})^m \sum_{i=0}^N \Big(\big(\frac{\lambda_{1,j-1}}{\lambda_{1j}}\big)^{i+1}
+ \big(\frac{\lambda_{1,j-1}}{\lambda_{1j}}\big)^{i+2}\Big) 
\leq C\frac{\delta_k}{\sigma}\lambda_{k+1}^{m}.
\end{split}
\end{equation*}
Recall the property (\ref{higher}). Denoting $c_{ij} = \bar a_{ij} (\sum_{j=1}^d G_{1j} )$,
Lemma \ref{lem_dec_def} implies that:
\begin{equation}\label{c_deco}
\sum_{j=1}^d G_{1j} = \hspace{-2mm} \sum_{i,j=2\ldots d,\; i\leq j}
\hspace{-2mm} c_{ij} \eta_{ij}\otimes\eta_{ij},
\end{equation}
where the coefficients $\{c_{ij}\}_{2\leq i\leq j}$ obey the same
bounds as those derived for $\sum_{j=1}^{d}G_{1j}$, namely: 
\begin{equation}\label{c_bd}
\|\nabla^{(m)} c_{ij}\|_0 \leq
C\frac{\delta_k}{\sigma}\lambda_{k+1}^{m} \quad\mbox{ for all } m
=0\ldots (K-k)(2N+4)-2N-1.
\end{equation}
In view of (\ref{Fbounds}), we note that if $\sigma\geq\sigma_0$
has been taken large enough, then:
$$\frac{C_k\delta_k}{2r_0} \leq (a_{ij})^2 - c_{ij} \leq \frac{4C_k\delta_k}{r_0}, $$
and therefore the following fields are well defined:
$$\{b_{ij}^{k+1}\in\mathcal{C}^{(K-k)(2N+4)-2N-1}(\bar\omega, \R)\}_{i,j=2\ldots d,\; i\leq j}, 
\qquad b_{ij}^{k+1} \doteq \big((a^{k+1}_{ij})^2 - c_{ij}\big )^{1/2}.$$
By (\ref{c_bd}) and (\ref{Fbounds}), we note that:
$$\|\nabla^{(m)}(b_{ij}^{k+1})^2\|_0\leq C\delta_k\mu_k^m + C
\frac{\delta_k}{\sigma}\lambda_{k+1}^{m}  \leq C\delta_k
\lambda_{k+1}^{m} \quad\mbox{ for all } m =0\ldots (K-k)(2N+4)-2N-1.$$
Applying Fa\'a di Bruno's formula, we also arrive at, for $m\neq 0$:
\begin{equation*}
\begin{split}
\|\nabla^{(m)}b_{ij}^{k+1}\|_0& \leq C \Big\|\sum_{p_1+2p_2+\ldots
  mp_m=m}\hspace{-3mm} (b_{ij}^{k+1})^{2(1/2-p_1-\ldots -p_m)}\prod_{t=1}^m
\big|\nabla^{(t)}(b ^{k+1}_{ij})^2\big|^{p_t}\Big\|_0 
\\  & \leq C \|b_{ij}^{k+1}\|_0\hspace{-3mm} \sum_{p_1+2p_2+\ldots mp_m=m} \prod_{t=1}^m
\Big(\frac{\delta_k\lambda_{k+1}^t}{C_k\delta_k /(2r_0)}\Big)^{p_t}
\leq C\delta_k^{1/2}\lambda_{k+1}^{m}. 
\end{split}
\end{equation*}
This already yields the bound \ref{Cbound22} at $k+1$:
\begin{equation}\label{Bbounds}
\|\nabla^{(m)}b_{ij}^{k+1}\|_0\leq C\delta_k^{1/2}\lambda_{k+1}^{m}
\quad\mbox{ for all } m=0\ldots (K-k)(2N+4)-2N-1.
\end{equation}
Finally, inserting (\ref{c_deco}) in (\ref{def3}), we take note of
the current defect decomposition:
\begin{equation}\label{def4}
\begin{split}
\mathcal{D}(A_0, V_{1d}, W_{1d}) & = \sum_{2\leq i\leq j} (a_{ij}^{k+1})^2\eta_{ij}\otimes\eta_{ij}
 - \sum_{j=1}^d G_{1j} + \Big(\mathcal{F} - \sum_{j=1}^d\mathcal{F}_{1j} \Big)
\\ & = \sum_{2\leq i\leq j} (b_{ij}^{k+1})^2\eta_{ij}\otimes\eta_{ij}
+\Big(\mathcal{F} - \sum_{j=1}^d\mathcal{F}_{1j} \Big).
\end{split}
\end{equation}

\medskip

{\bf 9. (Adding corrugations in the remaining codimensions $n=2\ldots \bar d$)}
We now recursively define the remaining vector fields 
$\{V_{n}\in\mathcal{C}^{(K-k)(2N+4)-2N-2}(\bar\omega,\R^{\bar d})\}_{n=2}^{\bar d}$,
$\{W_{n}\in\mathcal{C}^{(K-k)(2N+4)-2N-2}(\bar\omega,\R^{d})\}_{n=2}^{\bar d}$,
by first setting:
$$(V_{1}, W_{1}) \doteq (V_{1d}, W_{1d})$$
and then using (\ref{defi_per}) in Lemma \ref{lem_step2} for each $n=2\ldots \bar d$,
with $a=b_{ij}$ (we suppress the superscript $k+1$), $\lambda =
\mu_{k+1}$, and $\eta=\eta_{ij}$. We also define the one-to-one correspondence between the pairs of indices $ij$ and the single index $n$ as above, through the formula:
\begin{equation}\label{nij}
n = \frac{(i-2)(2d-1-i)}{2}+j, \qquad n=2\ldots \bar d, \quad
2\leq i\leq j,
\end{equation}
so that when $n=2$ then $ij= (2,2)$, and when $n=\bar{d}$ then $ij= (d,d)$. We declare:
\begin{equation}\label{VWp}
\begin{split}
& V_n\doteq V_{n-1}+\frac{b_{ij}}{\mu_{k+1}}\Gamma(\mu_{k+1} t_{\eta_{ij}})e_n,
\\ & W_n \doteq W_{n-1} - \frac{b_{ij}}{\mu_{k+1}} \Gamma(\mu_{k+1} t_{\eta_{ij}})\nabla V^n_{n-1}
+ \frac{(b_{ij})^2}{\mu_{k+1}} \bar\Gamma(\mu_{k+1}
t_{\eta_{ij}})\eta_{ij} \\ &\qquad \quad + 
\frac{b_{ij}}{\mu_{k+1}^2} \dbar\Gamma(\mu_{k+1}
t_{\eta_{ij}})\nabla b_{ij} + \Theta_n,
\end{split}
\end{equation}
where the additional perturbation term $\Theta_n$ in $W_{n}$ is:
\begin{equation}\label{omegap}
\begin{split}
& \Theta_n \doteq 2\frac{\Gamma_{1}(\mu_{k+1}t_{\eta_{ij}})}{\mu_{k+1}^{2}} b_{ij}
\sum_{s=1}^k \Gamma'(\mu_{s}t_{\eta_{ij}})\nabla b_{ij}^s
+  \frac{\Gamma_{1}(\mu_{k+1}t_{\eta_{ij}})}{\mu_{k+1}^2} b_{ij}  
\sum_{s=1}^k\mu_{s}\Gamma''(\mu_{s}t_{\eta_{ij}})b_{ij}^s\eta_{ij}\\ & \qquad\quad - 
\frac{\Gamma_{2}(\mu_{k+1}t_{\eta_{ij}})}{\mu_{k+1}^3} \nabla\Big(b_{ij}  
\sum_{s=1}^k\mu_{s}\Gamma''(\mu_{s}t_{\eta_{ij}})b_{ij}^s\Big).
\end{split}
\end{equation}
We set the principal fields at the consecutive counter $k+1$ as
\begin{equation}\label{final_step}
v_{k+1} \doteq V_{\bar d}, \quad W_{k+1} \doteq W_{\bar d},
\end{equation}
and use (\ref{step_err}) Lemma \ref{lem_step2} to get:
\begin{equation*}
\begin{split}
& \mathcal{D}(A_0, v_{k+1}, w_{k+1}) - \mathcal{D}(A_0, V_{1,d},
W_{1,d}) \\ & = -\sum_{n=2}^{\bar d}\Big(\big(\frac{1}{2}(\nabla V_{n})^T \nabla V_{n} +
\sym\nabla W_{n}\big) - \big(\frac{1}{2}(\nabla V_{n-1})^T \nabla V_{n-1} +
\sym\nabla W_{n-1}\big)\Big)\\ & = - \sum_{n=2}^{\bar d}
\Big((b_{ij})^2\eta_{ij}\otimes\eta_{ij} - 
\frac{\Gamma(\mu_{k+1}t_{\eta_{ij}})}{\mu_{k+1}} b_{ij}\nabla^{(2)}V_{n-1}^n 
+ \frac{\dbar\Gamma(\mu_{k+1}t_{\eta_{ij}})}{\mu_{k+1}^{2}}  
b_{ij}\nabla^{(2)} b_{ij} \\ & \qquad \qquad + \frac{\tbar\Gamma(\mu_{k+1} t_{\eta_{ij}})}{\mu_{k+1}^{2}}  
\nabla b_{ij}\otimes \nabla b_{ij} +\sym\nabla\Theta_n\Big).
\end{split}
\end{equation*}
Inserting the above in (\ref{def4}), we arrive at the final defect
decomposition bound:
\begin{equation}\label{def5}
\mathcal{D}(A_0, v_{k+1}, w_{k+1}) = \Big(\mathcal{F} -
\sum_{j=1}^d\mathcal{F}_{1j} \Big)  + \sum_{n=2}^{\bar d}\mathcal{H}_n - \mathcal{G},
\end{equation}
where the error fields $\{\mathcal{H}_n\}_{n=2}^{\bar d}$ and $\mathcal{G}$ are given in:
\begin{equation}\label{GH}
\begin{split}
& \mathcal{H}_n \doteq \frac{\Gamma(\mu_{k+1}t_{\eta_{ij}})}{\mu_{k+1}} b_{ij}\nabla^{(2)}V_{n-1}^n - \sym\nabla\Theta_n,\\
& \mathcal{G} \doteq \sum_{n=2}^{\bar d} \Big(\frac{\dbar\Gamma(\mu_{k+1}t_{\eta_{ij}})}{\mu_{k+1}^{2}}  
b_{ij}\nabla^{(2)} b_{ij} + \frac{\tbar\Gamma(\mu_{k+1} t_{\eta_{ij}})}{\mu_{k+1}^{2}}  
\nabla b_{ij}\otimes \nabla b_{ij}\Big).
\end{split}
\end{equation}

\medskip

{\bf 10. (Bounds on $V_n$)} Note first that by (\ref{final_step}) there holds:
\begin{equation*}
 v_{k+1} = V_{1d} + \sum_{n=2}^{\bar d} (V_n - V_{n-1}),
\qquad v_{k+1}^n=V_n^n = v_k^n +
\frac{\Gamma(\mu_{k+2}t_{\eta_{ij}})}{\mu_{k+1}} b_{ij} \quad\mbox{ for
 } n=2\ldots\bar d.
\end{equation*}
Now, in virtue of (\ref{Fbounds}), we estimate for all $m=0\ldots (K-k)(2N+4)-2N-2$:
$$\|\nabla^{(m)}(V_n-V_{n-1})\|_0\leq C\sum_{p+q=m}\mu_{k+1}^{p-1}\|\nabla^{(q)}b_{ij}\|_0
\leq C \sum_{p+q=m} \mu_{k+1}^{p-1}\delta_k^{1/2}\lambda_{k+1}^q\leq
C\delta_{k}^{1/2}\mu_{k+1}^{m-1}.$$
Applying (\ref{V1j-1}), (\ref{V1j-m}), and further \ref{pr_stima1},
\ref{Bbound12}, followed by \ref{pr_stima3}, \ref{Bbound32}, we observe:
\begin{equation*}
\begin{split}
& \|v_{k+1}^n-v^n\|_1\leq \|v_{k}-v\|_1+ C\delta_{k}^{1/2}\leq C\delta_0^{1/2},
\\ & \|\nabla^{(m+2)} v_{k+1}^n\|_0\leq \|\nabla^{(m+2)} v_{k}^n\|_0 +
C\delta_{k}^{1/2}\mu_{k+1}^{m+1} \\ & \hspace{2.5cm}\leq C\left\{\begin{array}{ll}
    \delta_{k-1}^{1/2}\mu_k^{m+1} & \mbox{if } k\geq 1 \\
    \delta_0^{1/2}\mu_0^{m+1} & \mbox{if } k = 0 
\end{array}\right\} +  C\delta_{k}^{1/2}\mu_{k+1}^{m+1} \leq
C\delta_{k}^{1/2}\mu_{k+1}^{m+1} \\ 
& \hspace{2.5cm}\mbox{for all } m=0\ldots (K-k)(2N+4)-2N-4, \quad 
n=2\ldots \bar d.
\end{split}
\end{equation*}
In the last inequality above we argued as in the proof of (\ref{fin2}), utilizing (\ref{increase2}). 
By (\ref{fin1}) we hence conclude the two inductive bounds \ref{Bbound12},
\ref{Bbound32} at the counter $k+1$.

\medskip

{\bf 11. (Bounds on $W_n$)} Similarly, as for $V_n$, we write by (\ref{final_step}):
$$w_{k+1} = W_{1d} + \sum_{n=2}^{\bar d} (W_n-W_{n-1}).$$
Recalling the definition (\ref{VWp})  and noting that $V_{n-1}^n = v_k^n$, we write:
\begin{equation}\label{Wp_diff}
\begin{split}
& \|\nabla^{(m)}(W_n-W_{n-1})\|_0\leq  C \hspace{-3mm}
\sum_{p+q+r=m}  \hspace{-3mm}\mu_{k+1}^{p-1}\|\nabla^{(q)}b_{ij}\|_0
\|\nabla^{(r+1)}v_k^n\|_0 \\ & + C\sum_{p+q=m} \mu_{k+1}^{p-1}\|\nabla^{(q)}(b_{ij})^2\|_0
+ C  \hspace{-3mm}\sum_{p+q+r=m}  \hspace{-3mm}\mu_{k+1}^{p-2}\|\nabla^{(q)}b_{ij}\|_0
\|\nabla^{(r+1)}b_{ij}\|_0 + \|\nabla^{(m)}\Theta_n\|_0.
\end{split}
\end{equation}
To estimate the first three terms in the right hand side, we use
(\ref{Bbounds}), \ref{pr_stima1}, \ref{pr_stima3}, together with
the induction assumptions \ref{Bbound12}, \ref{Bbound32}. Namely, for
all $m=0\ldots (K-k)(2N+4)-2N-2$:
\begin{equation*}
\begin{split}
& \sum_{p+q+r=m} \mu_{k+1}^{p-1}\|\nabla^{(q)}b_{ij}\|_0
\|\nabla^{(r+1)}v_k^n\|_0 \\ & \leq C\sum_{p+q=m} \mu_{k+1}^{p-1}\delta_k^{1/2}\lambda_{k+1}^q
\big(\delta_0^{1/2}+\|\nabla v\|_0\big)
+ C \hspace{-4mm}\sum_{p+q+r=m,\; r\geq 1} 
\hspace{-3mm}\mu_{k+1}^{p-1}\delta_k^{1/2}\lambda_{k+1}^q \delta_{k-1}^{1/2}\mu_k^r
\\ & \leq C \delta_k^{1/2}\mu_{k+1}^{m-1}\big(\delta_0^{1/2}+\|\nabla v\|_0\big) 
+ C \delta_k \mu_{k+1}^{m-1}
\\ & \leq C \delta_k^{1/2}\mu_{k+1}^{m-1}\big(1+\|\nabla v\|_0\big), 
\end{split}
\end{equation*}
where we also applied the fact that $\|\nabla^{(m+2)}v_k^n\|_0 \leq C
\delta_{k-1}^{1/2}\mu_k^{m+1} \leq C\delta_{k}^{1/2}\mu_{k+1}^{m+1}$
for all $k\geq 1$ by (\ref{increase2}), 
whereas at $k=0$ the same bound holds trivially in view of \ref{pr_stima3}.
Similarly:
\begin{equation*}
\begin{split}
& \sum_{p+q=m} \mu_{k+1}^{p-1}\|\nabla^{(q)}(b_{ij})^2\|_0
+ \sum_{p+q+r=m} \mu_{k+1}^{p-2}\|\nabla^{(q)}b_{ij}\|_0 
\|\nabla^{(r+1)}b_{ij}\|_0 \\ & \leq C \sum_{p+q=m}
\mu_{k+1}^{p-1}\delta_k\lambda_{k+1}^q + C \sum_{p+q+r=m}
\mu_{k+1}^{p-2}\delta_k\lambda_{k+1}^{q+r+1}\leq C\delta_k\mu_{k+1}^{m-1}.
\end{split}
\end{equation*}
Also, by definition (\ref{omegap}) and the inductive assumption \ref{Cbound22}:
\begin{equation*}
\begin{split}
\|\nabla^{(m)}\Theta_n\|_0\leq C\sum_{s=1}^k\Big( &
\sum_{p+q+r+t=m} \mu_{k+1}^{p-2}\|\nabla^{(q)}
b_{ij}\|_0\mu_s^r\|\nabla^{(t+1)}b_{ij}^s\|_0 \\ & + 
\hspace{-3mm}\sum_{p+q+r+t=m} \hspace{-3mm}\mu_{k+1}^{p-2}\|\nabla^{(q)}
b_{ij}\|_0\mu_s^{r+1}\|\nabla^{(t)}b_{ij}^s\|_0 \\ & + \hspace{-0mm}\sum_{p+q=m} \; \sum_{r+t+u=q+1} 
\hspace{-2mm}\mu_{k+1}^{p-3}\|\nabla^{(r)}b_{ij}\|_0\mu_s^{t+1}\|\nabla^{(u)}b_{ij}^s\|_0  \Big),
\end{split}
\end{equation*}
whereas, for all $m=0\ldots (K-k)(2N+4)-2N-2$:
\begin{equation*}
\begin{split}
\|\nabla^{(m)}\Theta_n\|_0 & \leq C\sum_{s=1}^k\Big(
\sum_{p+q+r+t=m} \mu_{k+1}^{p-2} \delta_k^{1/2} \lambda_{k+1}^q\mu_s^r
\delta_{s-1}^{1/2} \lambda_{s}^{t+1}+ 
\sum_{p+q+r+t=m} \mu_{k+1}^{p-2}\delta_k^{1/2} \lambda_{k+1}^q
\mu_s^{r+1}\delta_{s-1}^{1/2} \lambda_{s}^t \\ & \hspace{1.7cm} + \sum_{p+q=m} \;\sum_{r+t+u=q+1} 
\mu_{k+1}^{p-3} \delta_k^{1/2} \lambda_{k+1}^r\mu_s^{t+1}
\delta_{s-1}^{1/2} \lambda_{s}^u\Big)
\\ & \leq C \mu_{k+1}^{m-1} \sum_{s=1}^k \delta_k^{1/2}\delta_{s-1}^{1/2}\Big(
\frac{\lambda_s}{\mu_{k+1}} + \frac{\mu_s}{\mu_{k+1}} + \frac{\lambda_{k+1}\mu_s}{\mu_{k+1}^2}\Big) 
\\ & \leq C \delta_k \mu_{k+1}^{m-1}\sum_{s=1}^k \sigma^{(k-s+1)N/2}
\frac{\mu_s}{\mu_{k+1}} \leq C\delta_k\mu_{k+1}^{m-1},
\end{split}
\end{equation*}
as  $s=1\ldots k$, there holds $\mu_{k+1}/\mu_s\geq \sigma^{(k+1-s)N/2}$.
In summary, (\ref{Wp_diff}) becomes:
\begin{equation*}
\begin{split}
\|\nabla^{(m)}(W_n-W_{n-1})\|_0& \leq C
\delta_k^{1/2}\mu_{k+1}^{m-1}\big(1+\|\nabla v\|_0\big)  
\\ & \mbox{for all } m=0\ldots (K-k)(2N+4)-2N-2.
\end{split}
\end{equation*}
It now follows, through (\ref{fin2}), the inductive bounds
\ref{Bbound12}, \ref{Bbound42} at the counter value  $k+1$:
\begin{equation*}
\begin{split}
& \|w_{k+1}-w\|_1\leq \|w_k-w\|_1 + \sum_{n=2}^{\bar d} \|W_n-W_{n-1}\|_1\leq 
C \delta_0^{1/2}\big(1+\|\nabla v\|_0\big),
\\ & \|\nabla^{(m+2)}w_{k+1}\|_0\leq \|\nabla^{(m+2)}W_{1d}\|_0 + 
\sum_{n=2}^{\bar d} \|\nabla^{(m+2)}(W_n-W_{n-1})\|_0 \\ &
\hspace{2.56cm} \leq C\delta_k^{1/2}\mu_{k+1}^{m+1}\big(1+\|\nabla v\|_0\big)  
 \quad \mbox{for all } m =0\ldots (K-k)(2N+4)-2N-4.
\end{split}
\end{equation*}

\medskip

{\bf 12. (Bounds on $\mathcal{G}$ and $\mathcal{H}_n$)}
Recall the definition (\ref{GH}). From (\ref{Bbounds}) we directly see that:
\begin{equation}\label{Gbound}
\begin{split}
\|\nabla^{(m)}\mathcal{G}\|_0\leq C\sum_{p+q=m}
\mu_{k+1}^{p-2}\delta_k\lambda_{k+1}^{q+2} & \leq C
\delta_k\mu_{k+1}^m\big(\frac{\lambda_{k+1}}{\mu_{k+1}}\big)^2 \leq
C\frac{\delta_k}{\sigma^N}\mu_{k+1}^m 
\\ & \mbox{for all } m= 0\ldots (K-k)(2N+4)-2N-3.. 
\end{split}
\end{equation}
Since $V^n_{n-1}= v_k^n$, and since, by our construction:
$$v_k^n = v_0^n + \sum_{s=1}^k  
\frac{\Gamma(\mu_s t_{\eta_{ij}})}{\mu_s}b_{ij}^s,$$ 
we can rewrite the formula for $\mathcal{H}_n$ as:
\begin{equation}\label{H_again}
\begin{split}
\mathcal{H}_n & = \frac{\Gamma(\mu_{k+1}t_{\eta_{ij}})}{\mu_{k+1}} b_{ij}\nabla^{(2)}
\Big(v_0^n + \sum_{s=1}^k \frac{\Gamma(\mu_s t_{\eta_{ij}})}{\mu_j^s}b_{ij}^s\Big) -\sym\nabla\Theta_n
\\ & = \frac{\Gamma(\mu_{k+1}t_{\eta_{ij}})}{\mu_{k+1}} b_{ij}\nabla^{(2)}v_0^n 
+ \sum_{s=1}^k \frac{\Gamma(\mu_{k+1}t_{\eta_{ij}})}{\mu_{k+1}} b_{ij}
\frac{\Gamma(\mu_s t_{\eta_{ij}})}{\mu_s}\nabla^{(2)}b_{ij}^s 
\\ & \quad + 2 \sum_{s=1}^k \frac{\Gamma(\mu_{k+1}t_{\eta_{ij}})}{\mu_{k+1}} b_{ij}
\Gamma'(\mu_s t_{\eta_{ij}})\,\sym\big(\nabla b_{ij}^s\otimes \eta_{ij}\big)
\\ & \quad + \sum_{s=1}^k \frac{\Gamma(\mu_{k+1}t_{\eta_{ij}})}{\mu_{k+1}} b_{ij}
\mu_s \Gamma''(\mu_s t_{\eta_{ij}})b_{ij}^s\eta_{ij}\otimes \eta_{ij}
-\sym\nabla \Theta_n.
\end{split}
\end{equation}
The components of the third term in the right hand side above can be
decomposed as:
\begin{equation*}
\begin{split}
& \frac{\Gamma(\mu_{k+1}t_{\eta_{ij}})}{\mu_{k+1}} b_{ij}
\Gamma'(\mu_s t_{\eta_{ij}})\,\sym\big(\nabla b_{ij}^s\otimes
\eta_{ij}\big) \\ & = 
\sym\nabla\Big(\frac{\Gamma_1(\mu_{k+1}t_{\eta_{ij}})}{\mu_{k+1}^2} b_{ij} 
\Gamma'(\mu_s t_{\eta_{ij}})\nabla b_{ij}^s\Big)  -
\frac{\Gamma_1(\mu_{k+1}t_{\eta_{ij}})}{\mu_{k+1}^2} \sym\nabla\Big(b_{ij}
\Gamma'(\mu_s t_{\eta_{ij}})(\nabla b_{ij}^s\Big),
\end{split}
\end{equation*}
and, similarly, we express components of the fourth term as:
\begin{equation*}
\begin{split}
& \frac{\Gamma(\mu_{k+1}t_{\eta_{ij}})}{\mu_{k+1}} b_{ij}
\mu_s \Gamma''(\mu_s t_{\eta_{ij}})b_{ij}^s\eta_{ij}\otimes \eta_{ij}
\\ & = \sym\nabla\Big(\frac{\Gamma_1(\mu_{k+1}t_{\eta_{ij}})}{\mu_{k+1}^2} b_{ij}
\mu_s \Gamma''(\mu_s t_{\eta_{ij}})b_{ij}^s\eta_{ij}\Big) 
- \frac{\Gamma_1(\mu_{k+1}t_{\eta_{ij}})}{\mu_{k+1}^2} \sym\Big(\nabla \big(b_{ij}
\mu_s \Gamma''(\mu_s t_{\eta_{ij}})b_{ij}^s\big)\otimes \eta_{ij}\Big)
\\ & = \sym\nabla\Big(\frac{\Gamma_1(\mu_{k+1}t_{\eta_{ij}})}{\mu_{k+1}^2} b_{ij}
\mu_s \Gamma''(\mu_s t_{\eta_{ij}})b_{ij}^s\eta_{ij}\Big) 
- \sym\nabla\Big(\frac{\Gamma_2(\mu_{k+1}t_{\eta_{ij}})}{\mu_{k+1}^3} \nabla \big(b_{ij}
\mu_s \Gamma''(\mu_s t_{\eta_{ij}})b_{ij}^s\big)\Big) 
\\ & \quad + \frac{\Gamma_2(\mu_{k+1}t_{\eta_{ij}})}{\mu_{k+1}^3}
\nabla^{(2)}\Big( b_{ij}\mu_s \Gamma''(\mu_s t_{\eta_{ij}})b_{ij}^s\Big),
\end{split}
\end{equation*}
Recalling the formula for $\Theta_n$ in (\ref{omegap}), the equation
(\ref{H_again}) hence becomes:
\begin{equation}\label{H_again2}
\begin{split}
\mathcal{H}_n & = \frac{\Gamma(\mu_{k+1}t_{\eta_{ij}})}{\mu_{k+1}} b_{ij}\nabla^{(2)}v_0^n 
+ \sum_{s=1}^k \frac{\Gamma(\mu_{k+1}t_{\eta_{ij}})}{\mu_{k+1}} b_{ij}
\frac{\Gamma(\mu_s t_{\eta_{ij}})}{\mu_s}\nabla^{(2)}b_{ij}^s 
\\ & \quad - 2 \sum_{s=1}^k
\frac{\Gamma_1(\mu_{k+1}t_{\eta_{ij}})}{\mu_{k+1}^2} \sym\nabla\Big(b_{ij}
\Gamma'(\mu_s t_{\eta_{ij}})\nabla b_{ij}^s\Big)
\\ & \quad + \sum_{s=1}^k \frac{\Gamma_2(\mu_{k+1}t_{\eta_{ij}})}{\mu_{k+1}^3}
\nabla^{(2)}\Big( b_{ij}\mu_s \Gamma''(\mu_s t_{\eta_{ij}})b_{ij}^s\Big).
\end{split}
\end{equation}
We finally estimate the four terms above, together with their derivatives of
orders $m=0\ldots (K-k)(2N+4)-2N-3$. For the first term, we use
(\ref{Bbounds}) and \ref{pr_stima3} to get:
\begin{equation*}
\begin{split}
& \big\|\nabla^{(m)}\Big(\frac{\Gamma(\mu_{k+1}t_{\eta_{ij}})}{\mu_{k+1}}
b_{ij}\nabla^{(2)}v_0^n \Big)\big\|_0 \leq
C\sum_{p+q+r=m}\mu_{k+1}^{p-1}\|\nabla^{(q)}b_{ij}\|_0
\|\nabla^{(r+2)}v_0\|_0 \\ & \leq C\sum_{p+q+r=m}\mu_{k+1}^{p-1}\delta_k^{1/2}\lambda_{k+1}^q
\delta_0^{1/2}\mu_0^{r+1}\leq C \delta_{k}\sigma^{Nk/2}\mu_{k+1}^m\frac{\mu_0}{\mu_{k+1}}
\leq C\frac{\delta_k}{\sigma^N}\mu_{k+1}^m,
\end{split}
\end{equation*}
in view of $\mu_{k+1}/\mu_0\geq\sigma^{(k+2)N/2}$. This
is the only estimate necessitating the larger increase from
$\lambda_{1}$ to $\mu_1$ than from all the subsequent $\lambda_{k+1}$ to $\mu_{k+1}$.
For the second, third and fourth terms in (\ref{H_again2}), we fix $s=1\ldots k$ and use
(\ref{Bbounds}) together with 
the inductive assumption \ref{Cbound22}, to bound the derivatives of
$b_{ij}^s$. Thus, the components of the second term obey:
\begin{equation*}
\begin{split}
& \big\|\nabla^{(m)}\Big(\frac{\Gamma(\mu_{k+1}t_{\eta_{ij}})}{\mu_{k+1}} b_{ij}
\frac{\Gamma(\mu_s t_{\eta_{ij}})}{\mu_s}\nabla^{(2)}b_{ij}^s
\Big)\big\|_0\leq C\sum_{p+q+r+t=m}\mu_{k+1}^{p-1}\|\nabla^{(q)}b_{ij}\|_0
\mu_s^r \|\nabla^{(t+2)}b^s_{ij}\|_0\\ &  \leq
C\sum_{p+q+r+t=m}\mu_{k+1}^{p-1} \delta_k^{1/2}\lambda_{k+1}^q\mu_s^{r-1}\delta_{s-1}^{1/2}
\lambda_s^{t+2} \leq C\delta_k
\sigma^{(k-s+1)N/2}\mu_{k+1}^m\frac{\lambda_s^2}{\mu_{k+1}
  \mu_s}\leq C\frac{\delta_k}{\sigma^N}\mu_{k+1}^m,
\end{split}
\end{equation*}
as $\mu_{k+1}\mu_s/ \lambda_s^2\geq \sigma^{N/2} \sigma^{((k+1)-(s-1))N/2}=
\sigma^{(k-s+3)N/2}$.
For the third term in (\ref{H_again2}), we get:
\begin{equation*}
\begin{split}
& \big\|\nabla^{(m)}\Big(\frac{\Gamma_1(\mu_{k+1}t_{\eta_{ij}})}{\mu_{k+1}^2} \sym\nabla\big(b_{ij}
\Gamma'(\mu_s t_{\eta_{ij}})\nabla b_{ij}^s\big) \Big)\big\|_0 \\ & \leq C\sum_{p+q=m}
\; \sum_{r+t+u=q+1} \mu_{k+1}^{p-2}\|\nabla^{(r)}b_{ij}\|_0
\mu_s^t \|\nabla^{(u+1)}b^s_{ij}\|_0 \\ & \leq C\sum_{p+q=m}\;
\sum_{r+t+u=q+1} \mu_{k+1}^{p-2}\delta_k^{1/2}\lambda_{k+1}^r\mu_s^t\delta_{s-1}^{1/2}\lambda_s^{u+1}
\leq C \delta_k^{1/2}\delta_{s-1}^{1/2} \sum_{p+q=m}\mu_{k+1}^{p-2}\lambda_{k+1}^{q+1}\lambda_s
\\ & \leq C \delta_k \sigma^{(k-s+1)N/2}\mu_{k+1}^m\frac{\lambda_{k+1}\lambda_s}{\mu_{k+1}^2}
\leq C\frac{\delta_k}{\sigma^N}\mu_{k+1}^m,
\end{split}
\end{equation*}
as $\mu_{k+1}^2/(\lambda_{k+1}\lambda_s)\geq 
\sigma^{N/2}\sigma^{((k+1)-(s-1))N/2}= \sigma^{(k-s+3)N/2}$.
For the last term  in (\ref{H_again2}), we get:
\begin{equation*}
\begin{split}
& \big\|\nabla^{(m)}\Big(\frac{\Gamma_2(\mu_{k+1}t_{\eta_{ij}})}{\mu_{k+1}^3}
\nabla^{(2)}\big( b_{ij}\mu_s \Gamma''(\mu_s t_{\eta_{ij}})b_{ij}^s\big)
\Big)\big\|_0\\ & \leq C\sum_{p+q=m}\; \sum_{r+t+u=q+2} \mu_{k+1}^{p-3}\|\nabla^{(r)}b_{ij}\|_0
\mu_s^{t+1} \|\nabla^{u}b^s_{ij}\|_0 \\ & \leq C\sum_{p+q=m}
\; \sum_{r+t+u=q+2} \mu_{k+1}^{p-3}
\delta_k^{1/2}\lambda_{k+1}^r\mu_s^{t+1} \delta_{s-1}^{1/2}\lambda_s^u 
\leq C \delta_k^{1/2}\delta_{s-1}^{1/2} \sum_{p+q=m}\mu_{k+1}^{p-3}\lambda_{k+1}^{q+2}\mu_s
\\ & \leq C \delta_k \sigma^{(k-s+1)N/2}\mu_{k+1}^m\frac{\lambda_{k+1}^2\mu_s}{\mu_{k+1}^3}
\leq C\frac{\delta_k}{\sigma^N}\mu_{k+1}^m,
\end{split}
\end{equation*}
since  $\mu_{k+1}^3/(\lambda_{k+1}^2\mu_s)\geq 
\sigma^{N}\sigma^{((k+1)-s))N/2}= \sigma^{(k-s+3)N/2}$. This finally implies:
\begin{equation}\label{Hbound}
\|\nabla^{(m)}\mathcal{H}_n\|_0\leq C
\frac{\delta_k}{\sigma^N}\mu_{k+1}^m \quad \mbox{ for all } m =0\ldots (K-k)(2N+4)-2N-3.
\end{equation}
In view of the decomposition (\ref{def5}) and  the bounds on the
derivatives of the fields $\mathcal{F}$,
$\{\mathcal{F}_{1j}\}_{j=1}^d$, $\mathcal{G}$, $\{\mathcal{H}_n\}_{n=2}^{\bar d}$, 
given in (\ref{Fbounds}), (\ref{F1j_bound}), (\ref{Gbound}),
(\ref{Hbound}), we arrive at:
\begin{equation*}
\begin{split}
\|\nabla^{(m)}\mathcal{D}_{k+1}\|_0\leq C\Big(
\frac{\delta_k}{\sigma^{2N}}\mu_k^m + \frac{\delta_k}{\sigma^{N+2}}\lambda_{k+1}^m 
& + \frac{\delta_k}{\sigma^N}\mu_{k+1}^m \Big)\leq
C\frac{\delta_k}{\sigma^N}\mu_{k+1}^m  \\ & \mbox{ for all } m =0\ldots (K-k)(2N+4)-2N-3,
\end{split}
\end{equation*}
proving the final inductive estimate \ref{Bbound52} at the counter $k+1$.

\medskip

{\bf 13. (End of proof)} After the total of $K$ steps, we declare:
$$\tilde v \doteq v_K, \qquad \tilde w \doteq w_K, \qquad \tilde{\mathcal{D}} = 
(A-A_0) + \mathcal{D}_K = A
- \big(\frac{1}{2}(\nabla \tilde v)^T\nabla \tilde v + \sym\nabla \tilde w\big).$$
Recalling \ref{Bbound12}-\ref{Bbound42} and (\ref{DEF}), we conclude that:
\begin{align*}
& \|\tilde v - v\|_1\leq C\delta_0^{1/2} = C\|\mathcal{D}\|_0^{1/2}, \\
& \|\tilde w-w\|_1\leq C\delta_0^{1/2} \big(1+ \|\nabla v\|_0\big) =
C\|\mathcal{D}\|_0^{1/2} \big(1+ \|\nabla v\|_0\big),\\
& \|\nabla^{(2)}\tilde v\|_0\leq C\delta_{K-1}^{1/2}\mu_K = 
C\frac{\delta_0^{1/2}}{\sigma^{(K-1)N/2}}\mu_0\sigma^{d+N+(d+\frac{N}{2})(K-1)}
\\ & \hspace{3.7cm}= C\delta_0^{1/2}\mu_0 \sigma^{d+N+d(K-1)} = C\mathcal{M}\sigma^{dK+N},\\ 
& \|\nabla^{(2} \tilde w\|_0\leq C\delta_{K-1}^{1/2}\mu_K\big(1+\|\nabla v\|_0\big)
= C\mathcal{M}\sigma^{dK+N}\big(1+\|\nabla v\|_0\big),
\end{align*}
while \ref{pr_stima2} and \ref{Bbound52} yield:
\begin{align*}
\|\tilde{\mathcal{D}}\|_0 & \leq \|A-A_0\|_0 +\|\mathcal{D}_K\|_0\leq 
C_A\frac{\|A\|_{r,\beta}}{\mathcal{M}^{r+\beta}} \delta_0^{(r+\beta)/2}
+ C\delta_K \\ & = (C+C_A)\Big( \frac{\|A\|_{r,\beta}}{\mathcal{M}^{r+\beta}}
\|\mathcal{D}\|_0^{(r+\beta)/2} + \frac{\|\mathcal{D}\|_0}{\sigma^{NK}}\Big).
\end{align*}
These are precisely the claimed bounds \ref{Abound12}-\ref{Abound32}. The proof is done.
\endproof

\section{The Nash-Kuiper scheme and a proof of Theorem \ref{th_final}}\label{sec4}

The proof of Theorem \ref{th_final} relies on iterating Theorem
\ref{thm_stage} via "induction on stages" which is the essence of 
the so-called {\em Nash-Kuiper scheme}. We now quote the result in \cite[Theorem 5.1]{lew_conv},
which is similar to the argument in \cite[section 6]{CDS}. Recall our notation:
$$\mathcal{D}(A, v, w) \doteq A-\big(\frac{1}{2}(\nabla v)^T\nabla v + \sym\nabla w\big).$$

\begin{theorem}\label{th_NK}
Let $\omega\subset\R^d$ be an open, bounded, $d$-dimensional set, and
let $A\in\mathcal{C}^{r,\beta}(\bar\omega, \R^{d\times d}_\sym)$ with
$r\in \{0,1\}$, $\beta\in (0,1]$, $r+\beta<2$.
Assume that $k, J, S\geq 1$ are such that the statement of Theorem \ref{thm_stage}
holds true with $k$ replacing the specific codimension $\bar d$, $J$ replacing the exponent $dK+N$ in
\ref{Abound22}, and $S$ replacing the exponent $KN$ in \ref{Abound32}, namely:

\begin{equation*}
\left[\quad  \begin{minipage}{12cm}
There exists $\sigma_0>1$ depending only on $\omega, k, J, S$ such
that, given any $v\in\mathcal{C}^2(\bar\omega,\R^k)$, $w\in\mathcal{C}^2(\bar\omega,\R^d)$
and any $\mathcal{M}$, $\sigma$ with:
\begin{equation*}
\begin{split}
& \mathcal{D}\doteq \mathcal{D}(A, v,w) \quad \mbox{ satisfies } \quad
0< \|\mathcal{D}\big\|_0\leq 1,\\
& \mathcal{M}\geq \max\big\{\|v\|_2, \|w\|_2, 1\big\}, \qquad \sigma\geq \sigma_0,
\end{split}
\end{equation*}
there exists $\tilde v\in\mathcal{C}^2(\bar\omega,\R^k)$,
$\tilde w\in\mathcal{C}^2(\bar\omega,\R^d)$, obeying the bounds:
\begin{equation*}
\begin{split}
& \|\tilde v - v\|_1\leq C\|\mathcal{D}\|_0^{1/2}, \qquad\qquad
\|\tilde w -w\|_1\leq C\|\mathcal{D}\|_0^{1/2} \big(1+ \|\nabla v\|_0\big), 
\vspace{3mm} \\
& \|\nabla^2\tilde v\|_0\leq C \mathcal{M}\sigma^{J}, \hspace{12.5mm}
\|\nabla^2\tilde w\|_0\leq C \mathcal{M}\sigma^J \big(1+\|\nabla v\|_0\big),  
\vspace{3mm} \\ 
& \displaystyle{\|\mathcal{D}(A, \tilde v, \tilde w)\|_0\leq
  C\Big(\frac{\|A\|_{r,\beta}}{\mathcal{M}^{r+\beta}}\|\mathcal{D}\|_0^{(r+\beta)/2}
  + \frac{\|\mathcal{D}\|_0}{\sigma^{S}}\Big)},
\end{split}
\end{equation*}
with constants $C$ depending only on $\omega, k, r, \beta, J, S$.
\end{minipage}\quad \right]
\end{equation*}
Then, for every $v\in\mathcal{C}^2(\bar\omega,\R^k)$,
$w\in\mathcal{C}^2(\bar\omega,\R^d)$ such that:
$$\mathcal{D}\doteq \mathcal{D}(A, v,w) \quad \mbox{ satisfies } \quad 0<
\|\mathcal{D}\big\|_0\leq 1,$$
and for every $\alpha$ in the range:
\begin{equation*}
0< \alpha <\min\Big\{\frac{r+\beta}{2},\frac{1}{1+2J/S}\Big\},
\end{equation*}
there exist $\tilde v\in\mathcal{C}^{1,\alpha}(\bar\omega,\R^k)$ and
$\tilde w\in\mathcal{C}^{1,\alpha}(\bar\omega,\R^d)$ with the following properties:
\begin{align*}
& \|\tilde v - v\|_1\leq C\|\mathcal{D}\|_0^{1/2}, \quad \|\tilde w -
w\|_1\leq C\|\mathcal{D}\|_0^{1/2}(1+\|\nabla v\|_0), \vspace{1mm}\\
& A-\big(\frac{1}{2}(\nabla \tilde v)^T\nabla \tilde v + \sym\nabla
\tilde w\big) =0 \quad\mbox{ in }\; \bar\omega, 
\end{align*}
where the constants $C$ depend only on $\omega, k, r,\beta, J, S$.
\end{theorem}

\noindent We remark that \cite[Theorem 5.1]{lew_conv} is stated as above, but
with $r=0$. The proof, allowing now for $r=1$, follows by
a straightforward inspection of the proof in \cite{lew_conv}. 

\medskip

\noindent Clearly, Theorem \ref{th_NK} and Theorem \ref{thm_stage} yield together
the next result below, because:
\begin{equation*}
\begin{split}
\frac{1}{1+2J/S} = \frac{S}{S+2J} = \frac{KN}{KN+ (dK+N)} & \to
1\quad \mbox{ as } \; K, N\to\infty.
\end{split}
\end{equation*}

\medskip

\begin{corollary}\label{th_NKH}
Let $\omega\subset\R^d$ be an open, bounded set, and
let $A\in\mathcal{C}^{r,\beta}(\bar\omega, \R^{d\times d}_\sym)$ with
$r\in \{0,1\}$, $\beta\in (0,1]$, $r+\beta<2$.
Fix $\alpha$ as in (\ref{VKrange}).
Then, given $v\in\mathcal{C}^2(\bar\omega,\R^{\bar d})$, $w\in\mathcal{C}^2(\bar\omega,\R^d)$
such that:
$$\mathcal{D}\doteq\mathcal{D}(A,v,w) \quad\mbox{ satisfies } \quad 0<\|\mathcal{D}\|_0\leq 1,$$
there exist  $\tilde v\in\mathcal{C}^2(\bar\omega,\R^{\bar d})$,
$\tilde w\in\mathcal{C}^2(\bar\omega,\R^d)$
with the following properties:
\begin{align*}
& \|\tilde v - v\|_1\leq C\|\mathcal{D}\|_0^{1/2}, \quad \|\tilde w -
w\|_1\leq C\|\mathcal{D}\|_0^{1/2}(1+\|\nabla v\|_0), \vspace{1mm}\\
& A-\big(\frac{1}{2}(\nabla \tilde v)^T\nabla \tilde v + \sym\nabla
\tilde w\big) =0 \quad\mbox{ in }\; \bar\omega, 
\end{align*}
The constants $C$ depend only
on $\omega, r, \beta$ and $\alpha$. 
\end{corollary}

\bigskip

\noindent The proof of Theorem \ref{th_final} proceeds in the same manner as \cite[Theorem 1.1]{lew_conv}: we first use the basic stage-like construction to reduce the defect magnitude $\|\mathcal{D}\|_0$
below $1$, after which Corollary \ref{th_NKH} yields the final result. \endproof

\section{Extension to the isometric immersion system }\label{sec_iso_imm}

Using the same ideas underlying the proof of Theorem \ref{th_final}, one
can show the following full flexibility result for the isometric
immersion system (\ref{II}):

\begin{theorem}\label{th_final2}
Let $\omega\subset\R^d$ be an open, bounded set, diffeomorphic to the unit ball $B_1\subset \R^d$, let  $g\in
\mathcal{C}^{r,\beta}(\bar\omega, \R^{d\times d}_{\sym,>})$  be 
a Riemannian metric on $\bar\omega$, and $k$ be a codimension satisfying:
$$k\geq \bar d.$$
Then, for every immersion $u\in \mathcal{C}^1(\bar\omega,\R^{d+k})$ such that:
$$(\nabla u)^T\nabla u < g \quad \mbox{ in } \; \bar\omega,$$
for every $\epsilon>0$, and for every regularity exponent $\alpha$ in:
\begin{equation*}
0<\alpha<\min\Big\{\frac{r+\beta}{2}, 1\Big\},
\end{equation*}
there exists $\tilde u\in \mathcal{C}^{1,\alpha}(\bar\omega,\R^{d+k})$ such that:
$$ \|\tilde u - u\|_0\leq \epsilon \quad\mbox{ and } \quad 
(\nabla \tilde u)^T\nabla \tilde u = g \quad \mbox{ in }\; \bar\omega.$$
\end{theorem}

\medskip

\noindent We now outline the differences between the Nash-Kuiper
scheme for the Monge-Amp\`ere system (\ref{MA}) and the isometric
immersion system (\ref{II}), arising in the proof of Theorem \ref{th_final2}: 

\begin{itemize}
\item[1.] A version of Theorem \ref{th_NK} remains valid; see
  \cite[Theorem 1.3]{lew_full2d2k} in case $d=k=2$. Thus, it 
suffices to show a conclusion similar to Theorem
\ref{thm_stage}. The defect is now defined as:
$$\mathcal{D}(g, u) = g - (\nabla u)^T\nabla u.$$
Roughly speaking, assuming that $\|\mathcal{D}\|_0$ is sufficiently
small, and given a sufficiently large $\sigma$, we construct an
updated immersion $\tilde u$ with the properties:
\begin{align}
    &\|\tilde u-u\|_1\leq C\|\mathcal{D}(g,u)\|_0^{1/2},\qquad
    \|\nabla^{(2)}\tilde u\|_0\leq C\sigma^{dK+N}, \nonumber\\
    &\|\mathcal{D}(g,\tilde u)\|_0\leq
    C\Big(\frac{\|\mathcal{D}(g,u)\|_0}{\sigma^{KN}} + \|g\|_{r,\beta} \Big),\nonumber
\end{align}
where constants $C$ depend on $g$, $N, K$ and also on $\|u\|_2$ and
the immersion constant of $u$. In practice, since
the system (\ref{II}) is fully nonlinear and one cannot take
advantage of the linear component $\sym \nabla w$ in (\ref{MA}),
the defect needs to be calculated and decreased with respect to a
scaled version of $H_0$, whereas the last estimate above takes the
following form with the following initial assumption (compare \cite[Theorem 1.2]{lew_full2d2k}):
$$\big\|\mathcal{D}\big(g-\frac{\delta}{\sigma^{NK}}H_0, \tilde u\big
)\big\|_0\leq \frac{r_0}{5}\frac{\delta}{\sigma^{NK}} + C
\|g\|_{r,\beta} \quad \mbox{with} \quad 
 \|\mathcal{D}(g-\delta H_0, u)\|_0\leq \frac{r_0}{4}\delta. $$

\smallskip

\item[2.] A version of Kuiper's corrugation in Lemma \ref{lem_step2} is as follows (see \cite{CHI2, lew_full2d2k}):
\begin{equation*}        
\begin{split}
& \tilde u= u + \frac{\Gamma(\lambda t_\eta)}{\lambda} a(x)E_u(x) + 
T_{u}(x) \Big(\frac{\bar\Gamma(\lambda t_\eta)}{\lambda}a(x)^2\eta +
\frac{\dbar\Gamma(\lambda t_\eta)}{\lambda^2}a(x)\nabla u(x) + \Theta\Big),
\\ & \mbox{where }\; T_u = (\nabla u)\big((\nabla u)^T\nabla u\big)^{-1},
\end{split}
\end{equation*}
with $\Gamma(t) = \sqrt{2}\sin t$, $\bar\Gamma(t) = -\frac{1}{4}\sin
(2t)$ and $\dbar\Gamma(t) = \frac{1}{4}\cos (2t)$. 
Note the vector $E_u$ which is unit normal to the surface $u(\omega)$, 
and replaces the previously constant codimension direction $e_n$; and the
$w$-corrugation in (\ref{defi_per}), appearing as the tangential
component of $\tilde u - u$, carried by the tangent frame
$T_u$. The additional perturbation $\Theta$ arises from the
anticipated application of Lemma \ref{lem_IBP}. The formula
(\ref{step_err}) consequently becomes: 
\begin{equation}\label{corru_II}
\begin{split}
& (\nabla\tilde u)^T\nabla \tilde u - (\nabla u)^T\nabla u
-a(x)^2\eta\otimes \eta \\
& = -2\frac{\Gamma(\lambda t_\eta)}{\lambda} a [\langle\partial_s\partial_tu,
E_u\rangle]_{s,t=1\ldots d} + 2\frac{\dbar\Gamma(\lambda
  t_\eta)}{\lambda^2} a\nabla^2a + \frac{\tbar\Gamma(\lambda
  t_\eta)}{\lambda^2}\nabla a\otimes\nabla a +\mathcal{R},
\end{split}
\end{equation}
where $\tbar\Gamma(t) = 1-\frac{1}{2}\cos(2t)$. 
The additional error term $\mathcal{R}$ involves derivatives of $a$,
$E_u$, $T_u$, and $\Theta$; see \cite[Lemma 2.5]{lew_full2d2k} for a
precise version of the above calculation without the $a \nabla u$
component in $\tilde u$. In practice, $\mathcal{R}$ is of higher
order, appearing in the proofs as powers of the defect measure $\delta$,
greater than $1$, see \cite[Lemma 3.1]{CHI2}. Since
$\delta$ can be assumed to be sufficiently small relative to a fixed but
possibly large power of $\|\nabla^{(2)} u\|_0$, the term $\mathcal{R}$
never competes with the three leading order terms in (\ref{corru_II}).

\smallskip

\item[3.] While translating steps 1-12 in the proof of Theorem \ref{thm_stage},
the intermediate immersions $\{u_k\}_{k=0}^K$ are constructed with the
uniform immersion constants. In this construction, the
codimension vectors $\{e_n\}_{n=1\ldots\bar d}$ are replaced by a frame
of normal vectors $\{E_u^n\}_{n=1\ldots\bar d}$. 
The initial frame, normal to the surface $u_0(\omega)$, can be
chosen independently as in \cite{CaoIn2024}. However, along the
iteration in $k=1\ldots K$, the subsequent frames must be updated relative to
the previous ones so that the differences 
$(E^n_{u_{k+1}} - E^n_{u_k})$, together with
their derivatives, satisfy bounds consistent with those on
$(\nabla u_{k+1} - \nabla u_k)$. This general update procedure has been
carried out in \cite[Section 3]{lew_full2d2k} and applies here with
minor modifications, due to $d,\bar d \geq 2$ rather
than the case $d = k = 2$. See also the discussion in \cite{In2026} in
the context of codimension $k=1$ isometric immersions. 

\end{itemize}

\end{document}